\documentclass[10pt]{amsart}
\usepackage{amssymb,amsfonts,lineno,amsthm,amscd,stmaryrd}
\usepackage[leqno]{amsmath}
\usepackage[numbers]{natbib}

\theoremstyle{plain}
\newtheorem{thm}{Theorem}
\newtheorem{cor}[thm]{Corollary}
\newtheorem{lem}[thm]{Lemma}
\newtheorem{prop}[thm]{Proposition}
\theoremstyle{definition}

\newtheorem{remark}[thm]{Remark}

\numberwithin{thm}{section}
\numberwithin{equation}{section}

\newcommand{\EQ}[1]{\eqref{#1}}\newcommand{\re}[1]{\eqref{#1}}

\newcounter{hypo}

\DeclareMathOperator{\trace}{trace}

\DeclareMathOperator{\dist}{dist}

\DeclareMathOperator{\divg}{div}

\newcommand{\R}{\ensuremath{\mathbb{R}}}
\newcommand{\rn}{\R^n}

\newcommand{\ep}{\varepsilon}
\newcommand{\Sy}{\ensuremath{\mathcal{S}_n}}
\newcommand{\sph}{\ensuremath{{S}^{n-1}}}
\newcommand{\puccisub}[2]{\mathcal{M}^-_{#1,#2}}
\newcommand{\Puccisub}[2]{\mathcal{M}^+_{#1,#2}}
\newcommand{\PucciSub}{\Puccisub{\elp}{\Elp}}
\newcommand{\pucciSub}{\puccisub{\elp}{\Elp}}
\newcommand{\pucci}{\pucciSub}
\newcommand{\Pucci}{\PucciSub}
\newcommand{\Q}[1]{Q[ #1 ]}

\newcommand{\Elp}{\Lambda}
\newcommand{\elp}{\lambda}

\newcommand{\radfun}[1]{\xi_{#1}}

\newcommand{\al}{{\alpha^*}}
\newcommand{\tilal}{{\widetilde\alpha^*}}

\newcommand{\tilphi}{{\widetilde\Phi}}
\def\Xint#1{\mathchoice
{\XXint\displaystyle\textstyle{#1}}%
{\XXint\textstyle\scriptstyle{#1}}%
{\XXint\scriptstyle
\scriptscriptstyle{#1}}%
{\XXint\scriptscriptstyle
\scriptscriptstyle{#1}}%
\!\int}
\def\XXint#1#2#3{{
\setbox0=\hbox{$#1{#2#3}{\int}$}
\vcenter{\hbox{$#2#3$}}\kern-.5\wd0}}
\def\dashint{\Xint-}

\begin{document}
\title[Nonexistence of positive solutions of elliptic inequalities]{Nonexistence of positive supersolutions of elliptic equations via the maximum principle}
\author{Scott N. Armstrong}
\address{Department of Mathematics\\ The University of Chicago\\ 5734 S. University Avenue
Chicago, Illinois 60637.}
\email{armstrong@math.uchicago.edu}
\author{Boyan Sirakov}
\address{UFR SEGMI, Universit\'e Paris 10\\
92001 Nanterre Cedex, France \\
and CAMS, EHESS \\
54 bd Raspail \\
75270 Paris Cedex 06, France}
\email{sirakov@ehess.fr}
\date{\today}

\begin{abstract}
We introduce a new method for proving the nonexistence of positive supersolutions of elliptic inequalities in unbounded domains of $\R^n$.  The simplicity and robustness of our maximum principle-based argument provides for its applicability to many elliptic inequalities and systems, including quasilinear operators such as the $p$-Laplacian, and nondivergence form fully nonlinear operators such as Bellman-Isaacs operators. Our method gives new and optimal results in terms of the nonlinear functions appearing in the inequalities, and applies to inequalities holding in the whole space as well as exterior domains and cone-like domains.
\end{abstract}

\maketitle

\section{Introduction}

A well-studied problem in the theory of the elliptic partial differential equations is that of determining for which nonnegative, nonlinear functions $f=f(s,x)$ there exists a positive solution or supersolution $u>0$ of the equation
\begin{equation}\label{princ}
-Q[u] = f(u,x),
\end{equation}
in some subset of $\R^n$; here $Q$ denotes a second-order elliptic differential operator. A model case is the semilinear inequality
\begin{equation}\label{modeleq}
-\Delta u \ge f(u),
\end{equation}
where $f$ is a positive continuous function defined on $(0,\infty)$. There is a vast literature on the problem of obtaining sufficient conditions on $f$ to ensure the nonexistence of positive supersolutions  of such equations, both in $\R^n$ and in subsets of $\R^n$, which encompasses many different choices of operators $Q$ and nonlinear functions $f$.

\medskip

In this paper we introduce a new method for proving the nonexistence of supersolutions in unbounded domains. It has the advantage of being both simple and robust, allowing us to prove new and essentially optimal results for wide classes of equations and systems of equations of type \EQ{princ}. In particular, we extend many of the previous Liouville results by substantially relaxing the hypotheses on $f$ required for nonexistence. Namely, we impose only ``local" conditions on the behavior of $f(s,x)$, near $s=0$ or $s=\infty$, and for large $|x|$. Furthermore, our approach unites many previously known but seemingly disparate results by demonstrating that they follow from essentially the same argument.

Our method depends only on properties related to the maximum principle which are shared by many elliptic operators for which the solvability of \EQ{princ} has been studied. Consequently, our technique applies to inequalities in both divergence and nondivergence forms, and interpreted in the appropriate (classical, weak Sobolev, or viscosity) sense.

\medskip

To give a flavor of our results, let us consider the differential inequality \EQ{modeleq} in an exterior domain $\R^n \setminus B$, $n\ge 2$, where $B\subset \R^n$ is any ball. Under only the hypotheses that $f:(0,\infty) \to (0,\infty)$ is  continuous, as well as
\begin{equation}\label{modelf}
0 < \liminf_{s\searrow 0} s^{-n/(n-2)}f(s) \leq \infty\qquad\qquad\mbox{ if }n\ge 3,\end{equation}
\begin{equation}\label{modelf1}
\liminf_{s\to\infty} e^{as} f(s) = \infty \quad \mbox{for each} \ a > 0\qquad\mbox{ if }n=2,
\end{equation}
for each $a>0$, we show that there does not exist a positive (classical, viscosity or weak Sobolev) solution of \EQ{modeleq}. Therefore in dimensions $n\geq 3$ it is only the behavior of $f(s)$ near $s=0$ that determines whether or not supersolutions  exist, while in dimension $n=2$ it is the behavior of $f(s)$ at infinity which determines solvability.
These results are sharp and new.

Furthermore, we will see that if the inequality \EQ{modeleq} is assumed to hold only on $\mathcal{C}\setminus B$ where $\mathcal{C}$ is  a proper cone of $\rn$,  then we must make assumptions on  $f$ both at zero and at infinity in order to obtain a nonexistence result. Specifically, we exhibit exponents $\sigma^-<1<\sigma^+$ such that  \EQ{modeleq} has no positive solutions  provided that $f:(0,\infty) \to (0,\infty)$ is  continuous and
\begin{equation} \label{intro-cones}
\liminf_{s\searrow 0} s^{-\sigma^+} f(s) > 0, \quad \mbox{and} \quad \liminf_{s\to \infty} s^{-\sigma^-} f(s) > 0.
\end{equation}

It is usually thought that the most precise results for equations in divergence form like \eqref{modeleq} are obtained by exploiting their integral formulation. A notable feature of this work is that we deduce new and optimal results for such equations by a method whose main ideas-- in particular the use of the quantitative strong maximum principle (see (H3) and Theorem~\ref{qsmp} below)-- originate primarily from the theory of elliptic equations in nondivergence form.

We now give a rough list of the properties we assume the operator $Q$ possesses, and on which our method relies: \begin{enumerate}
\item[(H1)] $Q$ satisfies a weak comparison principle;
\item[(H2)] the equations $-Q[\Phi] = 0$ and $-Q[\tilde \Phi]= 0$ have solutions in $\R^n\setminus \{0\}$ which are asymptotically homogeneous and positive (resp. negative) at infinity. Usually $\Phi$ and $\tilde \Phi$ are the \emph{fundamental solutions} of $Q$;
\item[(H3)] nonnegative solutions of $-Q[u]\ge h(x)\ge 0$ have a lower bound (on compact subsets of the underlying domain) in terms of the measure of a set on which $h$ is greater than a positive constant;
\item[(H4)] nonnegative solutions of $-Q[u]\ge 0$ satisfy a weak Harnack inequality, or at least a ``very weak" Harnack inequality; and
\item[(H5)] the operator $Q$ possesses some homogeneity.
\end{enumerate}
Specific details on these hypotheses and on some operators which satisfy them are given in Section \ref{operators}. These properties are verified for instance by quasilinear operators of $p$-Laplacian type with solutions interpreted in the weak Sobolev sense, and by fully nonlinear Isaacs operators with solutions interpreted in the viscosity sense.

\medskip

We now make the following deliberately vague assertion:

\medskip

{\it Suppose $Q$ has the properties (H1)-(H5) above, and the behavior of $f(s,x)$ near $s=0$ and/or $s=\infty$ compares appropriately with that of the functions $\Phi$ and $\tilde\Phi$ for large $|x|$. Then there does not exist a positive solution of the inequality \re{princ} on any exterior domain in $\rn$.}

We prove a very general (and rigorous) version of this assertion in Section~\ref{liouville}, see Theorem~\ref{FNL}. The above statement is optimal in the sense that if a model nonlinearity $f$ does not satisfy its hypotheses, then \re{princ} has positive supersolutions.

Obviously a nonexistence result in exterior domains implies nonexistence in $\R^n$ as well as the absence of singular supersolutions in $\R^n$ with arbitrary singularities in a bounded set. Another advantage of the technique we introduce here is that it applies very easily to \emph{systems} of inequalities in unbounded domains.

Let us now give a brief account of the previous results on the subject. Due to the large number of works in the linear and quasilinear settings, we make no attempt to create an exhaustive bibliography here. Much more complete accounts can be found in the book of Veron \cite{V} and the survey articles of Mitidieri and Pohozaev \cite{MP} and Kondratiev, Liskevich, and Sobol \cite{KLS}. Gidas \cite{G} gave a simple proof of the fact that the equation $-\Delta u = u^\sigma$ has no solutions in $\R^n$, provided $\sigma\le n/(n-2)$. Condition \EQ{modelf} appeared first in Ni and Serrin \cite{NS}, where the nonexistence of {\it decaying} radial solutions to some quasilinear inequalities like $-\Delta_p u\ge |x|^{-\gamma}u^\sigma$ in $\R^n$ for $\sigma\le (n-\gamma)(p-1)/(n-p)$ was proved. In two important papers,  Bidaut-Veron \cite{B1} and Bidaut-Veron and Pohozaev \cite{BVP} extended these results by dropping the restrictions on the behavior of a supersolution $u$ and by showing that the same results hold in exterior domains of $\rn$. For more nonexistence results for positive solutions of quasilinear inequalities with pure power right-hand sides, we refer to  Serrin and Zou \cite{SZ2}, Liskevich, Skrypnik, and Skrypnik \cite{LSS04}. Liouville-type results for semilinear inequalities in nondivergence form can be found in the work by Kondratiev, Liskevich, and Sobol \cite{KLS2}.
Extensions to quasilinear inequalities in conical domains have been studied for instance by Bandle and Levine \cite{BaL}, Bandle and Essen \cite{Bae}, Berestycki, Capuzzo-Dolcetta, and Nirenberg \cite{BCN}, and Kondratiev, Liskevich, and Moroz \cite{KLM}.

Fully nonlinear inequalities of the form $F(D^2u) \ge u^\sigma$, where $F$ is an Isaacs operator, were first studied by Cutri and Leoni \cite{CL}, and later by Felmer and Quaas \cite{FQ}, in the case of a rotationally invariant $F$ and a solution in the whole space (see also Capuzzo-Dolcetta and Cutri \cite{CDC}). These results were recently extended in \cite{AS}, by a different method, to arbitrary Isaacs operators and to exterior domains. In particular,  the inequality $F(D^2u) \ge u^\sigma$ has no positive solutions in any exterior domain in $\R^n$, provided that $\sigma\le (\alpha^*+2)/\alpha^*$ (or $\alpha^* \leq 0$), where $\alpha^*= \alpha^*(F)$ characterizes the homogeneity of the upward-pointing fundamental solution of the operator $F$ (as found in \cite{ASS}).

As far as systems of inequalities are concerned, Liouville results were obtained by Mitidieri \cite{M}, Serrin and Zou \cite{SZ1}, for the case of a whole space, Bidaut-Veron \cite{B2} for quasilinear systems in exterior domains, Birindelli and Mitidieri \cite{BM}, Laptev \cite{La} for systems in cones, and Quaas and Sirakov \cite{QS3} for fully nonlinear systems in the whole space. For elliptic systems, the literature is more sparse and concerns only systems with pure power right-hand sides such as the Lane-Emden system  $-\Delta u = v^\sigma$, $-\Delta v = u^\rho$.

Despite the great variety of approaches and methods, most of the previous results required a global hypothesis on the function $f$, namely that $f$ be a power function or a combination of power functions. A notable exception is the very recent work of D'Ambrosio and Mitidieri \cite{dAM}, who obtained various nonexistence results for  divergence-form quasilinear inequalities in the whole space with only a local hypothesis on the function $f(s)$ near $s=0$, as in \EQ{modelf}. Their method is based on sophisticated integral inequalities and requires that the inequality holds in the whole space.

Finally, we note that there is a large literature concerning Liouville results for solutions (not supersolutions) of equations of the form $-Q[u]=f(u)$ in $\R^n$, which started with the well-known work by Gidas and Spruck \cite{GS}. For instance, it is known that $-\Delta u =f(u)$ in $\R^n$ has no classical positive solutions provided $s^{-\frac{n+2}{n-2}}f(s)$ is an increasing function on $(0,\infty)$; see \cite{LZ} and the references therein. These deep and important results are  quite delicate, with the nonexistence range $(\frac{n}{n-2},\frac{n+2}{n-2})$ depending on the conformal invariance of the Laplacian, on the precise behavior of $f$ on the whole interval $(0,\infty)$, on the differential equality being verified in the whole space, as well as  on the solutions being classical.
\medskip

This paper is organized as follows. In Section \ref{sectwo} we present the main ideas by proving the Liouville result we stated  above in  the simple particular case of \EQ{modeleq} and $n\ge 3$. We collect some preliminary observations in Section \ref{operators}, including a precise list of the properties (H1)-(H5) above as well as some estimates for the minima of positive supersolutions of $-Q[u]\ge 0$ over annuli. Our main results for scalar equations in exterior domains are presented in Section~\ref{liouville}. We extend the results for equation \EQ{modeleq} to conical domains in Section \ref{cones}.  We conclude in Section~\ref{systems} with applications of our method to systems of inequalities.

%We refer to \cite{B1,B2,dA,dAM, M,SZ1,SZ2}.

\section{A simple semilinear inequality} \label{sectwo}
In this section, we illustrate our main ideas on the semilinear inequality
\begin{equation} \label{slinxs}
-\Delta u \ge f(u)
\end{equation}
in exterior domains in dimension $n\geq 3$, and under the assumption that the nonlinearity $f=f(s)$ is positive and continuous on $(0,\infty)$. We will show that the additional hypothesis
\begin{equation} \label{fpow}
\liminf_{s\searrow 0} s^{-n/(n-2)}f(s) >0
\end{equation}
implies that the inequality \EQ{slinxs} has no positive solution in any exterior domain. Notice that we impose no requirements on the behavior of $f(s)$ away from $s=0$, apart from continuity and positivity. In particular, $f$ may have arbitrary decay at infinity.

It is easily checked that for $q > n/(n-2)$, the function
$u(x) = c \left( 1 + |x|^2 \right)^{-1/(q-1)}$
is a smooth supersolution of $-\Delta u = u^q$ in $\R^n$, for each sufficiently small $c>0$. Moreover, the function $\tilde u (x) : = c_q |x|^{-2/(q-1)}$ is a \emph{solution} of the equation in $\R^n \setminus \{ 0 \}$, if the constant $c_q$ is chosen appropriately.
Notice $u$ and $\tilde u$ decay to zero as $|x|\to \infty$, so having a hypothesis on the behaviour of $f(s)$ as $s\to 0$ is unavoidable for a nonexistence result to hold. Thus the following theorem is seen to be optimal in a certain sense.

\begin{thm}\label{slthm}
Assume that $n\geq 3$ and the nonlinearity $f:(0,\infty) \to (0,\infty)$ is continuous and satisfies \EQ{fpow}. Then the differential inequality \re{slinxs} has no positive solution in any exterior domain of $\R^n$.
\end{thm}

We have left the statement of Theorem \ref{slthm} intentionally vague as to the notion of supersolution, since the result holds regardless of whether we consider supersolutions in the classical, weak, or viscosity sense.

Several easy facts regarding the Laplacian on annuli are required for the proof of Theorem \ref{slthm}, and we state them now.

The key ingredient in the proof of Theorem \ref{slthm} is the following ``quantitative" strong maximum principle.

\begin{lem}\label{kslap}
Assume $h \in L^\infty(B_3\setminus B_{1/2})$ is nonnegative, and $u\geq 0$ satisfies
\begin{equation*}
-\Delta u \geq h(x)  \quad \mbox{in} \ B_3\setminus B_{1/2}.
\end{equation*}
There exists a constant $\bar c>0$ depending only on $n$ such that  for each $A\subset B_2\setminus B_1$
$$
\inf_{B_2\setminus B_1}u \geq \bar c |A|\inf_A h.
$$
\end{lem}
\begin{remark} We  denote with $\inf_A u$  the essential infimum of $u$ on the set $A$.
\end{remark}

Lemma \ref{kslap} is a simple consequence of the fact that Green's function for the Laplacian with respect to any domain is strictly positive away from the boundary of the domain, which yields
\begin{equation*}
\inf_{B_2\setminus B_1}u \geq  c \int_{B_2\setminus B_1} h(x) \, dx,
\end{equation*}
for some $c> 0$ depending only on the dimension $n$. See for example \cite[Lemma 3.2]{BC} and the references there for more precise statements on the Laplacian.

To show that it is only the behaviour of $f$ near zero which determines whether supersolutions of \eqref{slinxs} exist, we use the following consequence of the mean value property.
\begin{lem}\label{vwhlap}
For every $0 < \nu < 1$, there exists a constant $\bar C = \bar C(n,\nu)>1$ such that for any positive superharmonic function $u$ in $B_3 \setminus \bar B_{1/2}$ and any $x_0 \in B_2 \setminus B_1$, we have
\begin{equation*}
\left| \left\{ u \leq \bar C u(x_0) \right\} \cap (B_2\setminus B_1) \right| \geq \nu \left|B_2 \setminus B_1\right|.
\end{equation*}
\end{lem}
We remark that Lemma \ref{vwhlap} is clearly weaker than the weak Harnack inequality.

\medskip

Applying the comparison principle to a positive superharmonic function and the fundamental solution $\Phi(x) = |x|^{2-n}$ of Laplace's equation yields the following simple lemma, which is well-known. For the reader's convenience, we recall an elementary proof. Here and throughout the paper, $C$ and $c$ denote positive constants which may change from line to line.

\begin{lem}\label{fundycmpl}
Suppose that $u > 0$ is superharmonic in an exterior domain $\Omega$ of $\R^n$, with $n\geq 3$. Then there are constants $C,c > 0$, depending only on $u$ and $\Omega$, such that
\begin{equation} \label{fundycmp}
c r^{2-n} \leq \inf_{B_{2r} \setminus B_r} u \leq C\quad \mbox{for every sufficiently large} \ r> 0.
\end{equation}
\end{lem}
\begin{proof}

Fix $r_0> 0$  such that $\R^n \setminus B_{r_0} \subset \Omega$. Select $c>0$ so small that $u\ge c\Phi$ in a neighbourhood of $\partial B_{r_0}$. Then for each $\varepsilon >0$, there exists $\bar R= \bar R(\varepsilon) > r_0$ such that $u+\varepsilon\ge \varepsilon\ge c\Phi$ in $\R^n\setminus B_{\bar R}$. Applying the maximum principle to
$$
-\Delta(u+\varepsilon)\ge 0= -\Delta(c\Phi)
$$
in $B_{R}\setminus B_{r_0}$, for each $R>\bar R(\ep)$, we conclude that $u+\varepsilon\ge c\Phi $ in $\R^n\setminus B_{r_0}$. Letting $\varepsilon\to 0$ we obtain $u\ge c\Phi$ in $\R^n\setminus B_{r_0}$, which gives the first inequality in \EQ{fundycmp}.

 For the second inequality in \EQ{fundycmp}, observe that for every $r> r_0$
\begin{equation*}
u(x) \geq \Psi_r(x):=\left( \inf_{B_{2r}\setminus B_r} u \right) \left( 1 - r_0^{n-2} |x|^{2-n} \right) \quad \mbox{for every} \ x \in \partial (B_r\setminus B_{r_0}),
\end{equation*}
as well as $-\Delta u \ge 0=-\Delta \Psi_r$ in $\R^n \setminus B_{r_0}$. By the maximum principle we deduce that $u \geq \Psi_r(x)$ in $B_r \setminus B_{r_0}$. In particular, for every $r> 2r_0$, we have
\begin{equation*}
\inf_{B_{4r_0}\setminus B_{2r_0}} u \geq \left( \inf_{B_{2r}\setminus B_r} u \right) \left( 1 - 2^{2-n}  \right),
\end{equation*}
which yields the second inequality in \EQ{fundycmp}.
\end{proof}

Let us now combine the three lemmas above into a proof of Theorem~\ref{slthm}.

\begin{proof}[{\bf Proof of Theorem \ref{slthm}}]
Let us suppose that $u> 0$ is a supersolution of \EQ{slinxs} in $\R^n \setminus B_{r_0}$, for some $r_0 > 0$.
For each $r> 2r_0$, denote $u_r (x) := u(rx)$ and observe that $u_r$ is a supersolution of
\begin{equation*}
-\Delta u_r \geq r^2 f(u_r) \quad \mbox{in} \ \R^n \setminus B_{r_0/r}.
\end{equation*}
For each $r> 2r_0$, define the quantity
\begin{equation*}
m(r): = \inf_{\bar B_2\setminus B_1} u_r = \inf_{\bar B_{2r} \setminus B_r} u.
\end{equation*}
 Set $A_r:=\{x\in B_2\setminus B_1: m(r)\le u_r(x)\le \bar C m(r)\}$, where $\bar C = \bar C(n,\frac{1}{2})> 1$ is as in Lemma \ref{vwhlap}. Then Lemma \ref{vwhlap} implies that
\begin{equation*}
|A_r|\ge (1/2)|B_2\setminus B_1|.
\end{equation*}
Thus applying Lemma \ref{kslap} with $A=A_r$ and $h(x):= r^2f(u_r(x))$ produces the estimate
\begin{align}\label{limmins1}
m(r) \geq \frac{1}{2} \bar cr^2 |B_2\setminus B_1|\min_{s\in\left[ m(r),\bar C m(r)\right]} f(s) \quad \mbox{for every} \ r\geq 2r_0,
\end{align}
where $\bar c>0$ is as in Lemma \ref{kslap}. By Lemma \ref{fundycmpl} $ m(r)$ is bounded, so
\begin{equation} \label{limmins}
\min_{\left[ m(r),\bar C m(r)\right]} f\le C r^{-2}m(r)\le Cr^{-2} \to 0\quad\mbox{ as }\; r\to \infty .
\end{equation}
Since $f$ is continuous and $f> 0$ on $(0,\infty)$, it follows immediately from $m(r)\le C$ that $
 m(r) \to 0$ as $r\to \infty$.
Hence if $r$ is sufficiently large, \EQ{fpow} and \EQ{limmins} imply
\begin{equation}\label{limmins2}
(m(r))^{n/(n-2)} \le \min_{\left[ m(r),\bar C m(r)\right]} f\le  C r^{-2}m(r).
\end{equation}
We may rewrite this inequality as
\begin{equation} \label{mupbnd}
m(r) \leq C r^{2-n} \quad \mbox{for every sufficiently large} \ r\geq 2r_0.
\end{equation}
Recall that by Lemma \ref{fundycmpl} we also have, for some $c> 0$,
\begin{equation} \label{mlwbnd}
m(r) \geq c r^{2-n} \quad \mbox{for every} \ r > 2r_0.
\end{equation} Let us now define the quantity
\begin{equation*}
\rho(r) : = \inf_{B_{2r}\setminus B_r} \frac{u}{\Phi} > 0, \quad \mbox{for} \ r > 2r_0,\qquad \Phi(x)= |x|^{2-n}.
\end{equation*}
Observe that for every $r>r_0$ and $\ep > 0$, we may choose $R>r$ large enough that
\begin{equation*}
u(x) + \ep \geq \rho(r) \Phi(x) \quad \mbox{on} \ \partial \left( B_{R} \setminus B_r \right).
\end{equation*}
By the maximum principle, $ u(x) + \ep \geq \rho(r) \Phi(x)$ in $\ B_{R} \setminus B_r$.
Sending $R\to \infty$ and then $\ep \to 0$, we discover that
\begin{equation}\label{ineqs1}
u(x) \geq \rho(r) \Phi(x) \quad \mbox{in} \ \R^n \setminus B_r,
\end{equation}
that is, $\rho(r) = \inf_{\rn\setminus B_r} u/\Phi$. Therefore the map $r \mapsto \rho(r)$ is nondecreasing. For every $r>2r_0$, define the function
\begin{equation*}
v_r(x) := u(rx) - \rho(r/2) \Phi(rx).
\end{equation*}
Observe that by \EQ{ineqs1} we have  $v_r \geq 0$ in $\R^n \setminus B_{1/2}$, and
\begin{equation*}
-\Delta v_r \geq r^2 f(u_{r}) \quad \mbox{in} \ \R^n \setminus B_{1/2}.
\end{equation*}
Using again Lemma \ref{kslap} with $A=A_r$ and $h(x):= r^2f(u_r(x))$, we deduce from \re{limmins2} and \re{mlwbnd} that
\begin{equation*}
\inf_{B_2\setminus B_1} v_r \geq a r^2 (m(r))^{n/(n-2)} \geq a r^{2-n} \quad \mbox{for every sufficiently large} \ r \geq 2r_0,
\end{equation*}
where $a> 0$ does not depend on $r$. In particular,
\begin{equation*}
u(rx) \geq \left( \rho(r/2) + a\right) \Phi(rx) \quad \mbox{on} \ B_2 \setminus B_1.
\end{equation*}
That is, $\rho(r) \geq \rho(r/2) + a$ for all sufficiently large $r$. Therefore we obtain that $\lim_{r\to \infty}\rho(r) = \infty$, which contradicts our inequality \EQ{mupbnd}.
\end{proof}

\begin{remark} Note that if instead of \re{fpow} we assumed the stronger hypothesis $\liminf_{s\searrow 0} s^{-\sigma}f(s) >0$ for some $\sigma<n/(n-2)$, then  \re{mupbnd} is replaced by $m(r) \le r^{-\beta}$ for $\beta=\frac{2}{\sigma-1}>n-2$, which immediately contradicts \re{mlwbnd}.
\end{remark}
\begin{remark}
If in addition to \re{fpow} we  assumed
\begin{equation}\label{localinf}
\liminf_{s\to\infty}f(s)>0,
\end{equation}
then we do not need Lemma \ref{vwhlap}, that is, we do not need to use a weak Harnack inequality. Indeed, we can repeat the proof above with $A=B_2\setminus B_1$, observing that \re{localinf} prevents $m(r)$ from going to infinity as $r\to\infty$, so $\displaystyle\min_{\left[ m(r),\infty\right)} f\to 0$ as $r\to \infty$.
\end{remark}

The proof of the following analogue of Theorem \ref{slthm} for two dimensions is postponed until Section \ref{liouville}, where we obtain it as a consequence of Corollary \ref{FNLm}.

\begin{thm}\label{fpowtwodeethm}
Let $f$ be a positive, continuous function on $(0,\infty)$ which satisfies
\begin{equation}\label{fpowtwodee}
\lim_{s\to \infty} e^{a s} f(s) = \infty \quad \mbox{for every} \ a >0.
\end{equation}
Then the inequality \EQ{slinxs} has no positive solution in any exterior domain of $\R^2$.
\end{thm}

Observe that \EQ{fpowtwodee} is a condition on $f(s)$ near $s=\infty$, as opposed to near zero. This difference from condition \EQ{fpow} is due to the behavior of the fundamental solution of Laplace's equation near infinity in dimension $n=2$ versus higher dimensions. See Section \ref{liouville} for a much more detailed study of this phenomenon. In cone-like domains, one must impose conditions on $f$ both near $s=0$ and $s=\infty$ to obtain the nonexistence of supersolutions, as we will see in Section \ref{cones}.

Theorem \ref{fpowtwodeethm} is also sharp. Indeed,  for any $a > 0$, the function
\begin{equation*}
u(x) := \textstyle\frac{2}{a}\left( \log|x| + \log\left( \log|x| \right) \right)
\end{equation*}
is a smooth positive solution of the equation
\begin{equation*}
-\Delta u = \textstyle\frac{2}{a} e^{-a u} \quad \mbox{in} \ \R^2 \setminus B_{3}.
\end{equation*}
Note that, as is well-known, there is no positive solution of $-\Delta u \geq 0$ in $\R^2 \setminus \{ 0 \}$, except for constant functions. See Theorem \ref{twodee} for a more general statement.

\section{Preliminaries}\label{operators}

\subsection{Several properties of supersolutions}
In this section we state in detail and comment on the hypotheses (H1)-(H5) under which we prove our main Liouville results. We also confirm that these hypotheses are satisfied by the $p$-Laplacian operator and fully nonlinear Isaacs operators.

Recall the \emph{$p$-Laplacian} is defined by
\begin{equation*}
\Delta_p u := \divg\!\left( |Du|^{p-2} Du \right), \quad 1 < p < \infty.
\end{equation*}
For the sake of simplicity, we do not consider more general quasilinear operators, although our techniques apply for instance to operators of the more general form $Q[u] = \divg(A(x,Du))$, with $A$ satisfying hypotheses (1.1)-(1.4) in \cite{Dama}.

\medskip

A \emph{uniformly elliptic Isaacs operator} $F$ is a function $F: \Sy \to \R$ satisfying the uniform ellipticity condition
\begin{equation} \label{unifel}
\qquad \pucci(M-N) \leq F(M) - F(N) \leq \Pucci(M-N) \quad \mbox{for all} \ M,N\in \Sy,
\end{equation}
and which is positively homogeneous of order one:
\begin{equation} \label{poshom}
F(tM) = t F(M) \quad \mbox{for all} \ t\geq 0, \ M \in \Sy.
\end{equation}
Here $\Sy$ is the set of $n$-by-$n$ symmetric matrices, and $\Pucci$ and $\pucci$ are the Pucci extremal operators defined for instance in \cite{CC}. Equivalent to \EQ{unifel} and \EQ{poshom} is the requirement that $F$ be an inf-sup or a sup-inf of linear uniformly elliptic operators
\begin{equation*}
F(D^2u) = \inf_{\alpha} \sup_\beta \trace\left(A^{\alpha\beta} D^2u\right)
\end{equation*}
over a collection of matrices $\{ A^{\alpha\beta} \}$ such that $\lambda I_n \leq A^{\alpha\beta} \leq \Lambda I_n$ for all $\alpha$ and $\beta$. Consult \cite{CC} for more on fully nonlinear, uniformly elliptic equations.

\medskip

Our notion of solution is chosen to suit the particular operator under consideration. The $p$-Laplacian is of divergence form, and thus we use the weak integral formulation. More precisely, a \emph{weak supersolution} of the quasilinear equation  \begin{equation}\label{plapfu}
-\mathrm{div}A(x,Du) =   f(u,x)
\end{equation}
in a domain $\Omega \subseteq \R^n$ is a function $u\in W^{1,p}_{\mathrm{loc}}(\Omega)$ with the property that for all nonnegative  $\varphi \in C^\infty_0(\Omega)$ we have
\begin{equation*}
\int_\Omega A(x,Du) \cdot D\varphi \, dx \geq \int_\Omega f(u,x) \varphi \, dx.
\end{equation*}
When $Q[u]=F(D^2u)$ for an Isaacs operator $F$, the appropriate weak notion of solution is that of viscosity solution. Namely, $u$ satisfies  the inequality
\begin{equation*}
-F(D^2u)\geq \ (\leq) \ f(u,x)
\end{equation*}
in the \emph{viscosity sense} in $\Omega$ if for each  $x_0\in \Omega$ and $\varphi\in C^2(\Omega)$  for which the map $x \mapsto u(x) - \varphi(x)$ has a local minimum (maximum) at $x_0$, we have
\begin{equation*}
-F(D^2\varphi(x_0)) \geq \ (\leq) \ f(u(x_0),x_0).
\end{equation*}
Henceforth, when we write a differential inequality such as $-Q[u] \geq f(u,x)$, we intend that it be interpreted in the appropriate sense.

\medskip

We now present a list of properties which these operators share and upon which our method is based. We will confirm below that the following hold in the case that $Q$ is the $p$-Laplacian operator or an Isaacs operator:

 \begin{itemize}
 \item[(H1)] $Q$ satisfies  {\it a weak comparison principle}: if $-Q[u]\leq 0 \leq -Q[v]$ in a bounded domain $\Omega$, and $u\leq v$ on $\partial \Omega$, then $u \leq v$ in $\Omega$;
 \item[(H2)] $Q$ has {\it fundamental solutions}: there exist functions $\Phi, \widetilde \Phi$ which satisfy $-\Q{\Phi} =0=-\Q{\tilphi}$ in $ \R^n \setminus\{0\}$, and $\Phi, \widetilde \Phi$ are approximately homogeneous in the sense of  \re{homo} below;
 \item[(H3)] $Q$ satisfies {\it a quantitative strong comparison principle}: with $w=0,\Phi, \tilde \Phi$, if $-Q[u]\ge \chi_\omega \ge 0 = -Q[w]$ in a bounded $\Omega$ and some compact subset $\omega\subset \Omega$ of positive measure, then $u>w+c_0$ in any $K\subset\subset \Omega$, where $c_0>0$ depends only on $Q,K,\Omega$, and a lower bound for $|\omega|$;
 \item[(H4)] $Q$ satisfies  {\it a very weak Harnack inequality}: if $-Q[u]\ge 0$  in a bounded $\Omega$ and $K\subset\subset \Omega$, then for each $0 < \tau < 1$ there exists $\bar C = \bar C(\tau,Q,K,\Omega)>1$ such that for any  point $x_0 \in K$, we have
$
\left| \left\{ u \leq \bar C u(x_0) \right\} \cap K \right| \geq \tau \left|K\right|$;
\item[(H5)] $Q$ \emph{has no zero order term} and \emph{possesses some homogeneity}: precisely, we have $Q[u+c]=Q[u]$ for each  $c\in \mathbb{R}$;   $Q[tu]= t^{p-1}Q[u]$ for some $p>1$ and every $t\ge 0$;  if $u$ satisfies $-Q[u]\ge f(u,x)$ in $\Omega$ and we set $u_r(x):= u(rx)$, then $-Q[u_r] \ge r^p f(rx,u_r)$ in $\Omega_r:=\Omega/r$.
\end{itemize}

The  hypotheses (H2) and (H5) can be weakened, as will be obvious from the proofs below. Namely, we can assume that if $u$ satisfies $-Q[u]\ge f(u,x)$ in $\Omega$, then $-Q_r [u_r] \ge r^p f(u_r,rx)$ on
 $\Omega_r$, for some operator $Q_r$ which satisfies the same hypotheses as $Q$, with constants independent of $r$; and  that for some $\beta>0$ the operator $Q^t[ u] := t^{-\beta}Q[tu]$ satisfies the same hypotheses as $Q$ with constants independent of $t>0$. We can also assume that the functions $\Phi, \tilde \Phi$ be only subsolutions in some exterior domain in $\rn$ (except for the last statement in Theorem \ref{twodee}).

\medskip

% Some of this is not true, for if it was we would have proved the theorems for these operators too:
%The properties (H1)-(H5) have been established for various types of elliptic operators, for instance nondivergence form extensions of the $p$-Laplacian, studied for instance in ... or integro-differential operators of the type considered in ... Thus we expect our arguments to yield nonexistence results for inequalities involving these operators.

Let us now recall that both the $p$-Laplacian and Isaacs operators satisfy conditions (H1)-(H5). We begin by recalling the weak comparison principle. For  $p$-Laplacian type opertors, we refer for example to \cite[Corollary 3.4.2]{PS}, while for Isaacs operators, this is a particular case of the results in \cite{CC,UG}.

\begin{prop}\label{wcp}
Let $Q$ denote the $p$-Laplacian or an Isaacs operator. Suppose that $\Omega$ is a bounded domain, and $u$ and $v$ satisfy the inequalities
\begin{equation*}
-Q[u] \leq 0 \leq -Q[v] \quad \mbox{in} \ \Omega,
\end{equation*}
and $u \leq v$ on $\partial\Omega$. Then $u\leq v$ in $\Omega$.
\end{prop}

Another important property for our purposes is the availability of solutions of $-Q[u]\leq 0$ with given behavior at infinity.  For $\alpha\in \R$, we denote
\begin{equation}\label{eq:radfun}
\radfun\alpha(x) := \left\{ \begin{array}{ll}
|x|^{-\alpha} & \mathrm{if} \ \ \alpha > 0, \\
- \log |x| & \mathrm{if} \ \ \alpha = 0, \\
- |x|^{-\alpha} & \mathrm{if} \ \ \alpha < 0.
\end{array}\right.
\end{equation}

\begin{prop}\label{fundysol} Let $Q$ denote the $p$-Laplacian or an Isaacs operator. Then there exist numbers $\al,\tilal\in (-1,\infty)$ and functions $\Phi,\tilphi$ such that
\begin{equation}\label{fundey}
-\Q{\Phi} =0=-\Q{\tilphi} \quad\mbox{in }\ \R^n \setminus\{0\},
\end{equation}
and for some positive constants $c,C > 0$,
\begin{equation}\label{homo}
\begin{split}
c\xi_\al \leq \Phi  \leq C \xi_\al, & \quad \mbox{if} \ \alpha^* \neq 0 \\
c\xi_{\widetilde\alpha^*} \leq -\widetilde\Phi  \leq C \xi_{\widetilde\alpha^*}, & \quad \mbox{if} \ \widetilde\alpha^* \neq 0 \\
-C + \xi_0 \leq \Phi \ (\mbox{resp.} \ -\widetilde\Phi) \leq C + \xi_0, & \quad \mbox{if} \ \alpha^* = 0 \ (\mbox{resp.} \ \widetilde\alpha^* = 0),
\end{split}
\end{equation}
\end{prop}

It is well-known (and can be easily checked) that the $p$-Laplacian satisfies the statement above with $\al=\tilal = (n-p)/(p-1)$ and $\Phi= -\widetilde\Phi =\xi_{(n-p)/(p-1)}$. For the reader interested in extending the results in this paper to more general quasilinear operators, we note that results on the existence and behavior of singular solutions of quasilinear equalities can be found in the classical work of Serrin \cite{S2}.

For Isaacs operators the question of existence, uniqueness, and properties of fundamental solutions was studied in detail in the recent work \cite{ASS}. In particular, the result above is a consequence of Theorem 1.2 in that paper. We remark that for nonlinear Isaacs operators we have $\al\neq \tilal$, except in very particular cases. This is due to the fact that Isaacs operators are not odd, in general. For the Pucci extremal operators, for example, we have
\begin{equation*}
\alpha^*\left( \pucci \right) = \widetilde\alpha^*\left( \Pucci \right) = \frac{\Lambda}{\lambda}(n-1) -1,
\end{equation*}
\begin{equation*}
\alpha^*\left( \Pucci \right) = \widetilde\alpha^*\left( \pucci \right) = \frac{\lambda}{\Lambda}(n-1) -1.
\end{equation*}

\medskip

Central for our method is the following quantitative (uniform) strong maximum principle. This result, while well-known (and fundamental to the regularity theory of linear elliptic equations developed by Krylov and Safonov, see Section 4 in \cite{K}) is surprisingly under-utilized in the theory of elliptic equations.

\begin{thm} \label{qsmp}
Let $Q$ denote the $p$-Laplacian or an Isaacs operator. Assume that $K$ and $A$ are compact subsets of a  bounded domain $\Omega\subseteq\R^n$, with $|A| > 0$. Suppose that $v$ is nonnegative in $\Omega$ and satisfies
\begin{equation*}
-Q[v] \geq \chi_A \quad \mbox{in} \ \Omega,
\end{equation*}
where $\chi_A$ denotes the characteristic function of $A$.
\begin{enumerate}\item Then there exists a constant $c_0=c_0(Q,|A|,\Omega,K)>0$ such that
\begin{equation*}
 v \geq  c_0\quad\mbox{on }\; K.
\end{equation*}
\item Suppose in addition that $v\ge \Phi\ge 0$ on $\partial \Omega$, where $\Phi$ is as in the previous theorem, and $0\not\in \Omega$. Then there exists a constant $c_0=c_0(Q,|A|,\Omega,K)>0$ such that
\begin{equation*}
 v \geq \Phi+ c_0\quad\mbox{on }\; K.
\end{equation*}
\end{enumerate}
\end{thm}
\begin{proof}
For an Isaacs operator, we have
$$
-\pucci(D^2v) \ge \chi_A \quad \mbox{and}\quad -\pucci(D^2(v-\Phi)) \ge \chi_A\quad\mbox{ in }\; \ \Omega,
$$
 and thus both (i) and (ii) are consequences of \cite[Chapter 4, Theorem 2]{K}, after an easy reduction to a linear equation (see for instance \cite[page 781]{QS3}).

Let us give a proof for the $p$-Laplacian. Suppose that (i) or (ii) is false so that there exists a sequence of compact subsets $A_j \subseteq \Omega$ with $\inf_j |A_j|>0$, and a sequence of positive functions $v_j$ such that $-\Delta_p v_j \ge \chi_{A_j}$ in $\Omega$ and
\begin{equation} \label{pass0}
\mbox{either} \quad  v_j(x_j) \to 0 \qquad \mbox{or} \quad v_j(x_j)-\Phi(x_j)\to 0 \quad \mbox{as} \ j \to \infty,
\end{equation}
for some sequence of points $x_j\in K$. Let $\tilde v_j$ solve the Dirichlet problem $$-\Delta_p \tilde v_j = \chi_{A_j}\;\mbox{ in }\;\Omega,\qquad \tilde v_j=0 \quad(\mbox{or }\; \tilde v_j = \Phi) \;\mbox{ on }\;\partial \Omega.$$
Then by Theorem \ref{wcp}, $v_j\ge \tilde v_j$ in $\Omega$, and so we can replace $v_j$ by $\tilde v_j$.  For all  $\varphi \in C^\infty_0(\Omega)$ we have
\begin{equation}\label{pass1}
\int_\Omega |Dv_j|^{p-2} Dv_j \cdot D\varphi \, dx = \int_{A_j} \varphi \,dx.
\end{equation}
According to the $C^{1,\alpha}$ estimates for the $p$-Laplace equation (see \cite{Tolk2,diBen,Lie}), we deduce that $v_j$ is bounded in $C^{1,\alpha}(\Omega)$ for some $\alpha>0$. Therefore we may extract a subsequence of $v_j$ which converges to a function $v_0$ in $C^1(\overline{\Omega})$. We may pass to limits in \EQ{pass1} to obtain $-\Delta_p v_0\ge 0$ in $\Omega$, as well as $v_0\ge 0$ on $\partial \Omega$ (or $v_0\ge \Phi$ on $\partial\Omega$).
By the strong maximum principle (see \cite{Tolk1}, or Theorem \ref{wh} below) we conclude that either $v_0\equiv 0$ or $v_0>0$ in $\Omega$. In the case (ii), by the strong comparison principle (see Theorem 1.4 in \cite{Dama}), either $v_0\equiv\Phi$ or $v_0>\Phi$ in $\Omega$. To apply the strong comparison principle here, we must note that the gradient of $\Phi= \xi_{(n-p)/(p-1)}$ never vanishes in $\Omega$.

By passing to limits in \EQ{pass0}, we obtain $v_0\equiv0$ in $\Omega$, or in the case (ii) $v_0\equiv\Phi$ in $\Omega$. In either case, $v_0$ is $p$-harmonic, so that a passage to the limit in \re{pass1} gives
\begin{equation*}
\lim_{j\to\infty} \int_{A_j} \varphi \,dx \to 0 \quad \mbox{as}\  j \to \infty,
\end{equation*}
for each $\varphi \in C^\infty_0(\Omega)$. By taking $\varphi \ge 1$ except on a very small subset of $\Omega$, this is easily seen to be a contradiction, according to $\inf_j |A_j| > 0$.
\end{proof}

The final ingredient of our proofs of the Liouville results is the weak Harnack inequality. For weak solutions of degenerate quasilinear equations it is due to Serrin \cite{S} and Trudinger \cite{T}. In the nondivergence framework it was proved by Krylov and Safonov for strong solutions (see \cite{K}), see also \cite{T2}, while for viscosity solutions of Isaacs equations it was obtained by Caffarelli \cite{C};  see also Theorem 4.8 in \cite{CC}.
\begin{thm}\label{wh}
Let $u\ge 0$ and $-Q[u] \geq 0 $ in a bounded domain $\Omega$, where $Q$ is the $p$-Laplacian or an Isaacs operator. Then there exists $\gamma>0$ depending only on $Q$ and $n$, such that for each compact $K\subset \Omega$ we have
\begin{equation*}
\left( \dashint_{K} u^\gamma\, dx \right)^{1/\gamma} \leq C \inf_{K} u,
\end{equation*}
for some positive constant $C$, which depends only on $n,Q,K,\Omega$.
\end{thm}
\begin{remark} In some cases the use of this theorem can be avoided, at the expense of strengthening the hypotheses on $f$, see for instance Remark 2.7 after the proof of Theorem \ref{slthm}.
\end{remark}
\begin{remark} We actually use only the following weaker result:  {\it for each $\gamma<1$ there exists a constant $\bar C = \bar C(n,Q,\gamma) > 1$ such that for any nonnegative weak supersolution $u$ of $-Q[u]\geq 0$ in the annulus $B_3\setminus B_{1/2}$, and any $x_0 \in B_2\setminus B_1$,}
\begin{equation*}
\left| \left\{ u \leq \bar C u(x_0) \right\} \cap (B_2\setminus B_1) \right| \geq \gamma \left|B_2 \setminus B_1\right|.
\end{equation*}
This is a consequence of the ``very weak" Harnack inequality, which states that for every $\gamma>0$, there exists a constant $\bar c = \bar c(n,Q,\gamma,\Omega, K) \in(0, 1)$ such that for any nonnegative weak supersolution $u$ of $-Q[u]\ge0$ in $\Omega$,
$$
\left| \left\{ u \ge1 \right\} \cap K \right| \geq \gamma\left|K\right| \qquad\mbox{implies}\qquad u\ge \bar c\;\mbox{ in }K.$$
This fact, though a consequence of the weak Harnack inequality, is interesting in its own right. For instance, it admits a proof which is considerably simpler than the proof of the weak Harnack inequality while being sufficient to imply the H\"older estimates for solutions of $-Q[u]=0$.
\end{remark}

{\it The reader is advised that in  the rest of the paper only properties (H1)-(H5) will be used. In other words, the Liouville theorems stated in Section \ref{liouville} are proved for any $Q$ such that the inequalities $-Q[u]\geq (\leq) f(x,u)$ can be interpreted in such a way that properties (H1)--(H5), or  a subset of them, are satisfied.}

\subsection{Properties of minima of supersolutions on annuli}
%This is bad:
%From now on we assume \re{metaeq} has a nontrivial nonnegative (and hence strictly positive, by the strong maximum principle) solution in an exterior domain $\rn\setminus B_R$.

Our method for proving nonexistence theorems is based on the study of minima of supersolutions in annuli. In this section we obtain some preliminary estimates by comparing supersolutions of $-Q[u] \geq 0$ with the fundamental solutions of $Q$ from property (H2).

Note that, given $r_0>0$,  $\Phi$ and $\widetilde \Phi$ can be assumed to never vanish in $\rn\setminus B_{r_0}$, since if needed we simply add or subtract a constant from these functions. With this in mind, let us define the quantities
\begin{equation} \label{defrrho}
m(r) : = \inf_{B_{2r} \setminus B_r} u, \qquad \rho(r) := \inf_{B_{2r} \setminus B_r} \frac{u}{\Phi}, \qquad  \widetilde \rho(r):= \inf_{B_{2r} \setminus B_r} \frac{u}{\widetilde \Phi}.
\end{equation}

\begin{lem} \label{mbounds} Assume $Q$ satisfies (H1) and (H2).
Suppose that  $u\geq 0$ satisfies
\begin{equation*}
-Q[u] \geq 0 \quad \mbox{in} \  \R^n\setminus B_{r_0}.
\end{equation*}
Then for some $r_1> r_0$,
\begin{equation} \label{fntaub}
\begin{cases}
r\mapsto \rho(r) \ \mbox{is nondecreasing on} \ (r_0,\infty), & \mbox{if } \; \alpha^* > 0, \\
r\mapsto m(r) \ \mbox{is nondecreasing on} \ (r_0,\infty), & \mbox{if } \; \alpha^* \leq 0,\\
r\mapsto m(r) \ \mbox{is bounded on} \ (r_1,\infty), & \mbox{if } \; \widetilde \alpha^* > 0,\\
r\mapsto \widetilde \rho(r) \ \mbox{is bounded on} \ (r_1,\infty), & \mbox{if } \; \widetilde \alpha^* \leq 0.\\
\end{cases}
\end{equation}
\end{lem}
\begin{proof}
First consider the case  $\alpha^* > 0$. Then $\Phi > 0$ and $\Phi(x) \rightarrow 0$ as $|x|\to \infty$. Observe that for every $r>r_0$ and $\ep > 0$, we may choose $R> r$ large enough that
\begin{equation*}
u(x) + \ep \geq \rho(r) \Phi(x) \quad \mbox{on} \ \partial \left( B_R \setminus B_r \right).
\end{equation*}
By the weak comparison principle,
\begin{equation*}
u(x) + \ep \geq \rho(r) \Phi(x) \quad \mbox{in} \ B_R \setminus B_r.
\end{equation*}
Sending $R\to \infty$ and then $\ep \to 0$, we discover that
\begin{equation*}
u(x) \geq \rho(r) \Phi(x) \quad \mbox{in} \ \R^n \setminus B_r,\qquad\mbox{hence}\qquad \rho(r) = \inf_{\rn \setminus B_r} \frac{u}{\Phi}.
\end{equation*}
The desired monotonicity of $r\mapsto \rho(r)$ follows.

Next, suppose that $\alpha^*\leq 0$. Recall that  $\Phi(x) < 0$ for $|x|\geq r_0$ and $\Phi(x) \rightarrow -\infty$ as $|x|\to \infty$. Thus for every $r>r_0$ and $\delta > 0$, we can find $R> 0$ so large that
\begin{equation*}
u \geq m(r) + \delta \Phi \quad \mbox{on} \ \partial \left( B_R \setminus B_r \right).
\end{equation*}
Using the weak comparison principle and sending $R\to \infty$, we deduce that
\begin{equation*}
u \geq m(r)  +\delta \Phi \quad \mbox{in} \ \R^n \setminus B_r.
\end{equation*}
Now let $\delta \to 0$ to obtain $m(r) : = \inf_{\rn \setminus B_r} u$, and  the monotonicity of $r\mapsto m(r)$ on the interval $(r_0,\infty)$.

Suppose that $\widetilde\alpha^* > 0$. Then $\widetilde \Phi< 0$, and we may normalize $\widetilde \Phi$ so that
\begin{equation*}
\max_{\partial B_{r_0}} \widetilde \Phi = -1.
\end{equation*}
 For any $r>r_0$, we clearly have
\begin{equation*}
u(x) \geq \left( \inf_{B_{2r} \setminus B_r} u \right) \left( 1 + \widetilde \Phi(x) \right) \quad \mbox{for each} \ x \in \partial \left( B_r\setminus B_{r_0} \right).
\end{equation*}
By the weak comparison principle,
\begin{equation*}
u \geq \left( \inf_{B_{2r} \setminus B_r} u \right) \left( 1 + \widetilde \Phi \right) \quad \mbox{in} \ B_r\setminus B_{r_0},
\end{equation*}
for each $r>r_0$. Recalling \re{homo}, if $k$ is fixed  sufficiently large so that
\begin{equation*}
\inf_{B_{2kr_0}\setminus B_{kr_0}} u \geq \left( \inf_{B_{2r} \setminus B_r} u \right) \left( 1 - Ck^{-\widetilde \alpha^*} \right)\geq \frac{1}{2}\left( \inf_{B_{2r} \setminus B_r} u \right) ,
\end{equation*}
we obtain the third statement in \EQ{fntaub}, for $r\ge kr_0$.

Finally, we consider the case  $\widetilde \alpha^* \leq 0$.  Observe that
\begin{equation*}
u(x) \geq \widetilde \rho(r) \left( \widetilde \Phi(x) - \max_{\partial B_{r_0}} \widetilde \Phi \right) \quad \mbox{for each} \ x \in \partial \left( B_r\setminus B_{r_0} \right),
\end{equation*}
for any $r> r_0$. By the weak comparison principle,
\begin{equation*}
u(x) \geq \widetilde \rho(r) \left( \widetilde \Phi(x) - \max_{\partial B_{r_0}} \widetilde \Phi \right) \quad \mbox{in} \ B_r\setminus B_{r_0}.
\end{equation*}
Hence
\begin{equation*}
\inf_{B_{2kr_0}\setminus B_{kr_0}} u \geq \widetilde \rho(r) \left( \inf_{B_{2kr_0}\setminus B_{kr_0}} \widetilde \Phi - \max_{\partial B_{r_0}} \widetilde \Phi \right).
\end{equation*}
By fixing $k>1$ sufficiently large so that the quantity in the last parentheses is larger than one (recall we are in a case when $\widetilde\Phi(x)\to \infty$ as $|x|\to \infty$), the second part of \EQ{fntaub} follows. The lemma is proved.
\end{proof}

The following bounds on $m(r)$ are an immediate consequence of \re{fntaub}.

\begin{lem} \label{mbounds_1}
Assume $Q$ satisfies (H1) and (H2). Suppose that $r_0 > 0$ and $u\geq 0$ satisfy
\begin{equation*}
-Q[u] \geq 0 \quad \mbox{in} \  \R^n\setminus B_{r_0}.
\end{equation*}
Then for some $c,C>0$ depending on $Q$, $n$, $u$, and $r_0$, but not on $r$,
\begin{equation} \label{mbds}
\begin{cases}
m(r)\ge cr^{-\alpha^*} & \mbox{if} \ \alpha^* > 0, \\
m(r)\ge c & \mbox{if} \ \alpha^* \leq 0,\\
\end{cases}\qquad\mbox{and}\qquad
\begin{cases}
m(r)\le C & \mbox{if} \ \widetilde \alpha^* > 0,\\
m(r)\le C\log r & \mbox{if} \ \widetilde \alpha^* = 0,\\
m(r)\le Cr^{-\widetilde\alpha^*} & \mbox{if} \ \widetilde \alpha^* < 0.\\
\end{cases}
\end{equation}
\end{lem}

Finally, we observe that in some situations the map $r\mapsto m(r)$ is an nonincreasing function, in contrast with one of the conclusions of Lemma \ref{mbounds}.

\begin{lem} \label{mbounds_new} Assume $Q$ satisfies (H1) and (H2).
Suppose  $u\geq 0$ satisfies either
\begin{equation*}
-Q[u] \geq 0 \quad \mbox{in} \ B_R\qquad \mbox{or}\qquad \left\{ \begin{array}{l} -Q[u]\geq 0 \quad \mbox{in}\  B_R\setminus \{0\},\\
\widetilde \alpha^*\geq 0,\end{array}\right.
\end{equation*}
for some $R>0$. Then $r\to m(r)$ is nonincreasing on $(0,R)$.  \end{lem}

\begin{proof} The first statement is obvious, since the maximum principle implies that $m(r)=\inf_{B_r}u$. Let us prove the second statement. By subtracting a constant from $\widetilde\Phi$ if necessary, we may assume that $\widetilde \Phi < 0$ in $B_R\setminus \{ 0 \}$. Since $\widetilde \Phi(x) \to -\infty$ as $x \to 0$, for every $0 < r <R$ and $\delta > 0$, there exists $0 < \ep < r$ small enough that
\begin{equation*}
u \geq m(r) + \delta \widetilde \Phi \quad \mbox{on} \ \partial \left( B_r \setminus B_\ep \right).
\end{equation*}
By the weak comparison principle,
\begin{equation*}
u \geq m(r) + \delta \widetilde \Phi \quad \mbox{in} \ B_r \setminus B_\ep.
\end{equation*}
Sending $\ep \to 0$ and then $\delta \to 0$, we deduce that
\begin{equation*}
u \geq m(r) \quad \mbox{in} \ B_r \setminus \{ 0 \},\qquad\mbox{hence}\qquad m(r) = \inf_{B_r \setminus \{ 0 \}}u.
\end{equation*}
The monotonicity of the map $r\mapsto m(r)$ on the interval $(0,R)$ follows.
\end{proof}

\section{The Liouville Theorems}\label{liouville}

This section contains our main results on the nonexistence of solutions of elliptic inequalities. As we pointed out in the previous section, all results will be announced under some or all of the hypotheses (H1)-(H5) on the elliptic operator $Q$. These properties hold for weak solutions of quasilinear inequalities of $p$-Laplacian type, as well as for viscosity supersolutions of fully nonlinear equations of Isaacs type. Therefore all the following results will be valid for such supersolutions and inequalities.

We pursue with this choice of exposition to emphasize the independence of the method on the particular type of operators and weak solutions that we consider. We believe that modifications of our arguments will yield analogous results for inequalities involving elliptic operators with lower order terms as well as other types of operators, for instance mean-curvature-type operators, nondivergence form extensions of the $p$-Laplacian studied by Birindelli and Demengel \cite{BD}, nonlinear integral operators (c.f. \cite{CS}), and so on.

\subsection{Statement of the main result}
We begin by providing a brief overview of the main ideas in the proof of our main result, Theorem \ref{FNL}, which will also motivate the complicated hypotheses (f1)-(f4), below. Assume that we have a positive solution $u>0$ of the inequality
\begin{equation} \label{FNLeq}
-Q[u] \geq f(u,x) \quad\mbox{in} \ \R^n\setminus B_{r_0}.
\end{equation}
Setting $u_r(x)= u(rx)$ for $r\ge 2r_0$ and using (H5), we see that $u_r$ is a solution of
\begin{equation*}
-Q[u_r]\ge r^pf(u_r,rx)\ge0 \quad\mbox{in} \ \R^n \setminus B_{1/2},
\end{equation*}
where $p> 1$ is as in (H5). Then property (H4) implies that the set
\begin{equation*}
A_r := \{ x\in  B_2\setminus B_1 :  m(r)\leq u_r(x)\leq \bar C m(r)\}, \quad r> 2r_0,
\end{equation*}
is such that $|A_r|\ge (1/2)|B_2\setminus B_1|$, provided $\bar C >1$ is large enough. Then by (H3),
\begin{equation}\label{parteq1}
m(r) \geq \bar c r^{p}  \inf_{m(r)\le s\le \bar Cm(r), \ x\in A_r} f(s,rx).
\end{equation}
So it remains to discover hypotheses on $f$ which imply that \EQ{parteq1} is incompatible with the bounds on $m(r)$ obtained from Lemma \ref{mbounds_1}. First, if the simple nondegeneracy condition (f2) below is in force, then we immediately obtain from \EQ{parteq1} that either $m(r) \to 0$ or else $m(r) \to \infty$ as $r\to \infty$. We then impose conditions on $f$ to rule out both of these alternatives; these are, respectively, (f3) and (f4) below. In light of \EQ{parteq1}, we see that the former need concern only the behavior of $f(s,x)$ near $s=0$ and $|x|=\infty$, and the latter the behavior of $f(s,x)$ near $s=\infty$ and $|x|=\infty$.

\medskip

Our precise hypotheses on the function $f=f(s,x)$ are as follows:

\begin{enumerate}
\item[(f1)] $f: (0,\infty) \times (\R^n \setminus B_{r_0}) \to (0,\infty)$ is continuous;

\item[(f2)] $|x|^pf(s,x) \to \infty$ as $|x|\to \infty$ locally uniformly in $s\in (0,\infty)$;

\item[(f3)] either $\alpha^* \leq 0$, or else $\alpha^* > 0$ and there exists a constant $\mu > 0$ such that if we define
\begin{equation*}
\Psi_k(x) := |x|^p  \displaystyle\inf_{k\Phi(x)\le s\le \mu}s^{1-p} f(s,x) \quad\mbox{and}\quad h(k):=  \liminf_{|x|\to \infty} \Psi_k(x),
\end{equation*}
then $0 < h(k) \leq \infty$ for each $k>0$, and
\begin{equation*}
\lim_{k\to \infty} h(k) =\infty.
\end{equation*}

\item[(f4)] either $\widetilde \alpha^* > 0$, or else $\widetilde \alpha^*\leq 0$ and there exists a constant $\mu > 0$ such that  if we define
\begin{equation*}
\widetilde\Psi_k(x) := |x|^p \inf_{\mu\le s\le k\widetilde\Phi(x)} s^{1-p} f(s,x) \quad\mbox{and}\quad \widetilde h(k):=  \liminf_{|x|\to \infty} \widetilde\Psi_k(x),
\end{equation*}
then $0 < \widetilde h(k) \leq \infty$ for each $k>0$, and
\begin{equation*}
 \lim_{k\to 0} \widetilde h(k) =\infty.
\end{equation*}
\end{enumerate}
Observe that (f3) is void if $\alpha^* \leq 0$, while (f4) is void in the case  $\widetilde \alpha^* > 0$. We recall that for the $p$-Laplacian operator we have
\begin{equation*}
\alpha^* = \widetilde\alpha^* = \frac{n-p}{p-1},
\end{equation*}
while for an Isaacs operator with ellipticity $\Lambda/\lambda$, in general $\alpha^*\neq \widetilde\alpha^*$, and each of $\alpha^*$ and $\widetilde\alpha^*$ can be any number in the interval
\begin{equation*}
\left[ \frac{\lambda}{\Lambda}(n-1) - 1 , \frac{\Lambda}{\lambda}(n-1) - 1 \right].
\end{equation*}

Our main result is:

\begin{thm}\label{FNL}
Assume that $n\geq 2$, and $Q$ and $f$ satisfy (H1)-(H5) as well as (f1)-(f4), above. Then there does not exist a positive supersolution $u>0$ of \re{FNLeq}.
\end{thm}

We prove Theorem \ref{FNL} in the following subsection, and conclude the present one by stating a consequence for nonlinearities $f$ of the simpler form
\begin{equation*}
f(s,x) = |x|^{-\gamma}  g(s).
\end{equation*}
For such $f$, we observe at once that conditions (f1) and (f2) are together equivalent to the statement
\begin{equation}\label{f12ggam}
g:(0,\infty) \to (0,\infty) \quad \mbox{is continuous, and} \quad \gamma < p.
\end{equation}
We claim that, together with \EQ{f12ggam}, a sufficient condition for (f3) is
\begin{equation} \label{f3ggam}
\mbox{if} \ \alpha^* > 0, \quad \mbox{then} \ \liminf_{s\searrow 0} s^{-\sigma^*}g(s) > 0, \quad \mbox{for} \ \sigma^*:= (p-1) + \frac{p-\gamma}{\alpha^*}.
\end{equation}
Observe first that \EQ{f3ggam} implies that
\begin{equation*}
\eta= \eta(\mu):= \inf_{0 < s < \mu} s^{-\sigma^*} g(s) > 0, \quad \mbox{for every} \ \mu > 0.
\end{equation*}
Thus with $\Psi_k(x)$ as in (f3), we have for each $\mu > 0$ and all sufficiently large $|x|$,
\begin{equation*}
\Psi_k(x) = |x|^p \inf_{k\Phi(x) \leq s \leq \mu} s^{1-p} f(s,x) \geq \eta \inf_{k\Phi(x) \leq s \leq \mu} |x|^{p-\gamma} s^{1-p+\sigma^*}.
\end{equation*}
Since $1-p+\sigma^* = (p-\gamma) / \alpha^* > 0$, the last infimum above is attained at $s = k\Phi(x)$ (recall that $\Phi(x)\to 0$ as $|x|\to \infty$, since $\alpha^*>0$).  Hence we obtain, by the approximate homogeneity of $\Phi$, that
\begin{equation*}
\Psi_k(x) \geq c \eta k^{1-p+\sigma^*} |x|^{p-\gamma-\alpha^*(1-p + \sigma^*)} = \eta k^{1-p+\sigma^*} \rightarrow \infty \quad \mbox{as} \ k \to \infty,
\end{equation*}
where we have also used $1-p+\sigma^*-(p-\gamma)/\alpha^* = 0$. This confirms that (f3) holds. A similar analysis on the validity of  (f4) when $f(s,x) = |x|^{-\gamma}g(s)$ yields the following corollary, which contains Theorems \ref{slthm} and \ref{fpowtwodeethm} as very particular cases.

\begin{cor}\label{FNLm}
Assume that $n\geq 2$ and $Q$ satisfies (H1)-(H5). Then the differential inequality
\begin{equation*}
-Q[u] \geq |x|^{-\gamma} g(u),\qquad \gamma<p,
\end{equation*}
has no positive solution in any exterior domain, provided  the function $g:(0,\infty)\to(0,\infty)$ is continuous and satisfies
\begin{eqnarray} \label{f3ggam1}
\mbox{if} \ \alpha^* > 0, &\mbox{then}& \ \liminf_{s\searrow 0} s^{-\sigma^*}g(s) > 0, \quad \mbox{for} \ \sigma^*:= (p-1) + \frac{p-\gamma}{\alpha^*},\\
\label{f4ggam0}
\mbox{if}  \ \widetilde \alpha^* = 0,  &\mbox{then}& \ \liminf_{s\to \infty} e^{as} g(s) > 0, \quad \mbox{for every} \ a>0,\\
\label{f4ggam1}
\mbox{if} \ \widetilde \alpha^* < 0, &\mbox{then}& \ \liminf_{s\to \infty} s^{-\widetilde\sigma^*}g(s) > 0, \quad \mbox{for} \ \widetilde\sigma^*:= (p-1) + \frac{p-\gamma}{\widetilde\alpha^*}.
\end{eqnarray}
\end{cor}
Observe that
\begin{equation*}
-\infty < \widetilde\sigma^* < p-1 < \sigma^* < \infty.
\end{equation*}
Applied to the model nonlinearity $f(s,x) = |x|^{-\gamma}s^\sigma$, the conditions \EQ{f3ggam1}, \EQ{f4ggam0} and \EQ{f4ggam1} are sharp. Indeed, it was shown in \cite{AS} that the inequality
\begin{equation}\label{eqpow}
-Q[u] \ge |x|^{-\gamma} u^\sigma
\end{equation}
has a positive solution in $\R^n \setminus \{ 0 \}$ if $\alpha^*>0$ and $\sigma > \sigma^*$, and even in whole space $\R^n$ in the case $\gamma \leq 0$. The argument in \cite{AS} can be easily modified to show that \EQ{f4ggam0} and \EQ{f4ggam1} are similarly sharp. For example, in the case $\widetilde \alpha^* = 0$, then the function $u(x) = \widetilde\Phi(x) + \log \widetilde\Phi(x)$ is a supersolution of the inequality
\begin{equation*}
-Q[u] \geq ce^{-a u}
\end{equation*}
in some exterior domain, for some $a,c> 0$. We multiplying $u$ by a positive constant, we can have any $a > 0$ we wish. Similarly, if $\widetilde\alpha^*<0$ and $\sigma>\widetilde\sigma^*$ then multiplying some power of $\widetilde \Phi$ by a suitably chosen constant gives a solution of  \re{eqpow} in $\R^n \setminus \{ 0 \}$.

Notice also that we have $\widetilde \sigma^*>0$ when $p + \widetilde\alpha^*(p-1) < \gamma < p$. For such values of $\gamma$ and $\widetilde \alpha^*<0$, we see that there exist uniformly elliptic operators such that sublinear inequalities with nonlinearities that behave at infinity like $u^\sigma$,  $\sigma\in (0,\widetilde\sigma^*)$, may have positive solutions.

Finally, as mentioned above, both (f3) and (f4) are void in the case $\alpha^* \leq 0$ and $\widetilde\alpha^* > 0$, and we have the nonexistence of supersolutions in exterior domains under the modest hypotheses (f1) and (f2). In fact, in this case it is an immediate consequence of Lemmas \ref{mbounds} and \ref{mbounds_new} that we do not need any hypotheses apart from the nonnegativity of $f$, provided the inequality holds in $\R^n\setminus\{0\}$.
\begin{thm} \label{twodee}
Assume $Q$ satisfies (H1), (H2), and the usual strong maximum principle. Suppose  $u\geq 0$ satisfies either
\begin{equation*}
\left\{ \begin{array}{l} -Q[u]\ge 0 \quad \mbox{in}\  \rn,\\
 \alpha^*\le 0.\end{array}\right. \qquad \mbox{or}\qquad \left\{ \begin{array}{l} -Q[u]\ge 0 \quad \mbox{in}\  \rn\setminus \{0\},\\
\alpha^*\le 0, \;\widetilde \alpha^*\ge 0.\end{array}\right.
\end{equation*}
Then $u$ is constant.
\end{thm}
\begin{proof}
By adding a constant to $u$, we may suppose that $\inf u = 0$. According to Lemmas \ref{mbounds} and \ref{mbounds_new}, the map $r \mapsto m(r)$ is constant on $(0,\infty)$, and hence $m(r) \equiv 0$. The strong maximum principle then implies that $u\equiv 0$.
\end{proof}

The sharpness of Theorem \ref{twodee} illustrates the difference between nonexistence results in the whole space and in more general unbounded domains. For instance, the inequality $-\Delta_p u\ge 0$  has no positive solutions in $\rn$ for every $p\ge n$, while the same inequality has no positive solutions in the punctured space $\R^n \setminus \!\{ 0 \}$ only in the case $p=n$.

\subsection{Proof of Theorem \ref{FNL}}
To obtain a contradiction, let us suppose that $u > 0$ is a solution of the differential inequality
\begin{equation*}
-Q[u]\geq f(u,x) \quad \mbox{in} \ \R^n \setminus B_{r_0},
\end{equation*}
for some $r_0> 1$. For each $r> 2r_0$, denote $u_r(x) : = u(rx)$, and observe that (H5) says $u_r$ is a supersolution of
\begin{equation*}
-Q[u_r] \geq r^p f(u_r, rx) \quad \mbox{in} \ \R^n \setminus B_{1/2}.
\end{equation*}
As before, set $m(r) := \inf_{B_{2r}\setminus B_r} u = \inf_{B_2\setminus B_1} u_r$ for $r>2r_0$. Let $\bar C$ be as in (H4) with $\tau =1/2$, $K=\bar B_2\setminus B_1$, and $\Omega = B_3\setminus B_{1/2}$. According to (H4), for each $r>2r_0$ the set $A_r:=(B_{2} \setminus B_1 ) \cap \{ m(r) \leq u_r \leq \bar C m(r) \}$ has measure at least $\frac{1}{2} |B_2\setminus B_1|$. Then (H3)  and (H5) imply that for some $c> 0$,
\begin{equation}\label{mrandf}
{m(r)}^{p-1} \geq cr^p \inf\left\{ f(s,x) : r \leq |x| \leq 2r, \ m(r) \leq s \leq \bar C m(r) \right\}
\end{equation}
for every $r>2r_0$. Owing to hypothesis (f2), we immediately deduce that
\begin{equation} \label{noloitering}
\mbox{either} \quad m(r) \rightarrow 0 \quad \mbox{or} \quad m(r) \rightarrow + \infty \quad \mbox{as} \ r \to \infty.
\end{equation}
Indeed, if we had a subsequence $r_j \to \infty$ such that $m(r_j) \rightarrow a \in (0,\infty)$, then by sending $r=r_j \to \infty$ in \EQ{mrandf} we obtain a contradiction to (f2). We will complete the proof of Theorem \ref{FNL} by showing that the alternatives in \EQ{noloitering} are contradicted by (f3) and (f4), respectively.

\medskip

\noindent\emph{Case 1: $m(r) \to 0$ as $r \to \infty$.} If $\alpha^* \leq 0$, then we may immediately appeal to Lemma~\ref{mbounds_1} to obtain a contradiction. So we need only consider the case that $\alpha^* > 0$, for which Lemma \ref{mbounds_1} provides the lower bound $
m(r) \geq c r^{-\alpha^*}$ for all $r > 2r_0$.
We next establish the upper bound
\begin{equation} \label{mub}
m(r) \leq C r^{-\alpha^*} \quad \mbox{for all sufficiently large} \ r > 2r_0.
\end{equation}
Let $k> 0$ and $r> 2r_0$ be very large, and suppose that $m(r) \geq kr^{-\alpha^*}$. Then assuming that $r> 0$ is large enough that $\bar C m(r) \leq \mu$, and using \EQ{mrandf}, we obtain for some $C_1$ such that $C_1\Phi(x)\ge |x|^{-\alpha^*}$,
\begin{align*}
\inf_{r\leq|x|\leq 2r} \Psi_{C_1k}(x) & \leq \inf\left\{ s^{1-p}|x|^p f(s,x) : r\leq |x|\leq 2r, \ k |x|^{-\alpha^*} \leq s \leq \mu \right\} \\
& \leq \inf\left\{ s^{1-p} |x|^p f(s,x) : r \leq |x| \leq 2r, \ m(r) \leq s \leq \bar Cm(r) \right\} \\
& \leq C m(r)^{1-p}\inf \left\{ r^p  f(s,x) : r \leq |x| \leq 2r, \ m(r) \leq s \leq \bar Cm(r)  \right\} \\
& \leq C.
\end{align*}
Owing to (f3), this is clearly impossible if $k > 0$ and $r> 2r_0$ are large enough. Thus we obtain the upper bound \EQ{mub}, and we have the two-sided estimate
\begin{equation} \label{mrtrap}
c r^{-\alpha^*} \leq m(r) \leq C r^{-\alpha^*}
\end{equation}
for large $r>2r_0$.

According to Lemma \ref{mbounds}, the map $r \mapsto \rho(r)$ is nondecreasing. Thus for every $r>2r_0$, the functions
\begin{equation*}
v_r(x): = r^{\alpha^*}u(rx),\qquad w_r(x):=  \rho(r) r^{\alpha^*}   \Phi(rx)
\end{equation*}
satisfy $v_r \geq w_r$ in $\R^n \setminus B_{1}$, and we have
\begin{equation*}
-Q[v_r] \geq r^{p+\alpha^*(p-1)} f( u_r,rx)\ge 0 = -Q[w_r] \quad \mbox{in} \ \R^n \setminus B_{1/2}.
\end{equation*}
Note that $c\le v_r\le C$ on $A_r$, by \re{mrtrap}.
Using \re{homo}, \EQ{mrandf}, \EQ{mrtrap}, and (f3), for large enough $r>2r_0$ we have
\begin{align*}
\inf_{A_r} \left( r^{p+\alpha^*(p-1)} f(u_r,rx)\right)&\geq c r^p \inf_{A_r}  \left( \frac{f(r^{-\alpha^*}v_r, rx)}{(r^{-\alpha^*}v_r)^{p-1}} \right) \\
& \geq c \inf\left\{ s^{1-p}|y|^p f(s,y) : r \leq |y|, \ m(r) \leq s \leq  C m(r) \right\}  \\
& \geq c\inf_{x\ge r} \Psi_c(x)\geq c.
\end{align*} Hence for such $r$,
\begin{equation*}
-Q[c^{1/(1-p)}v_r] \geq  \chi_{A_r} \ge 0 = -Q[c^{1/(1-p)}w_r]\quad \mbox{in} \ B_{5} \setminus B_{1}.
\end{equation*}
According to (H3) and \re{homo}, this implies
\begin{equation*}
v_r \geq w_r + c_0 \geq (\rho(r) + c_1)\,  r^{\alpha^*}   \Phi(rx) \quad \mbox{in} \ B_{4}\setminus B_{2},
\end{equation*}
for some $c_1> 0$ which does not depend on $r$. Unwinding the definitions, we discover that $u \geq (\rho(r) + c_1) \Phi$ in $B_{4r} \setminus B_{2 r}$. In particular, $\rho(2 r) \geq \rho(r) + c_1$, and we deduce that $\rho(r) \rightarrow \infty$ as $r \to \infty$. This contradicts the second inequality in \EQ{mrtrap}, since obviously $\rho(r)\le Cm(r)\max_{r\le |x|\le 2r}\Phi\le Cm(r)r^{\alpha^*}$. The proof in the case $\lim_{r\to 0} m(r) = 0$ is complete.

\medskip

\noindent\emph{Case 2: $m(r) \to \infty$ as $r\to \infty$.} If $\widetilde \alpha^* > 0$, we obtain an immediate contradiction by applying Lemma \ref{mbounds_1}, so we may suppose that $\widetilde \alpha^*\leq 0$ and the second alternative in (f4) is in force. We may assume that $\widetilde\Phi$ is normalized so that $\max_{\partial B_{r_0}} \widetilde \Phi = 1$, as well as $\widetilde\Phi > 0$ on $\R^n \setminus B_{r_0}$.

Lemma \ref{mbounds_1}  gives the upper bound
\begin{equation} \label{mrupna}
m(r) \leq C \max_{r\le |x|\le 2r}\widetilde \Phi(x) \leq C \min_{r\le |x|\le 2r}\widetilde \Phi(x), \quad r> 2r_0,
\end{equation}
using the approximate homogeneity of $\widetilde\Phi$.

We next establish a lower bound for $m(r)$.
Let $k>0$ and assume  $r>2r_0$ is large enough that $m(r) \geq \mu$. Suppose for contradiction  that $m(r) \leq k \min_{r\le |x|\le 2r}\widetilde \Phi(x)$ for all large $r$. Then this, \EQ{mrandf}, and our assumption that $\lim_{r\to \infty} m(r)~=~\infty$ imply that for sufficiently large $r> 2r_0$
\begin{align*}
\inf_{r\le |x|\le 2r}\widetilde\Psi_{\bar C k}(x)
&\leq
\inf\left\{ s^{1-p} |x|^p f(s,x) : r\le |x| \leq 2 r , \ \mu \leq s \leq \bar C k \widetilde \Phi(x) \right\} \\
& \leq \inf \left\{ s^{1-p} r^p f(s,x) : r \leq |x| \leq 2r, \ m(r) \leq s \leq \bar C m(r) \right\} \\
 &\leq Cm(r)^{1-p} \inf \left\{  r^p f(s,x) : r \leq |x| \leq 2r, \ m(r) \leq s \leq \bar C m(r) \right\} \\
 &\leq C.
\end{align*}
This contradicts (f4) if $k> 0$ is  sufficiently small, and we have the lower bound
\begin{equation*}
c\max_{r\le |x|\le 2r} \widetilde\Phi(x) \leq c \min_{r\le |x|\le 2r} \widetilde\Phi(x) \leq m(r).
\end{equation*}
Recalling \EQ{mrupna}, we have the two-sided estimate
\begin{equation}\label{mrtrap2}
c\max_{r\le |x|\le 2r}\widetilde \Phi(x) \leq m(r) \leq C \min_{r\le |x|\le 2r}\widetilde \Phi(x)  \quad \mbox{for sufficiently large}  \ r>r_0.
\end{equation}
Define the quantity
\begin{equation*}
\omega(r) : = \inf_{B_{2r}\setminus B_r} \frac{u}{\widetilde\Phi -1}, \quad r > r_0.
\end{equation*}
By the weak comparison principle, we have that
\begin{equation*}
u \geq \omega(r) \left( \widetilde \Phi - 1 \right)\quad \mbox{in} \ B_{r} \setminus B_{r_0}.
\end{equation*}
Thus for $r> 2r_0$, the positive functions
\begin{equation*}
v_r(x): = u(rx) + \omega(2r), \qquad w_r:= \omega(2r)  \widetilde \Phi(rx)
\end{equation*}
are such that $v_r\ge w_r$ and satisfy the differential inequalities
\begin{equation*}
-Q[v_r] \geq r^{p} f(u_r,rx)\ge 0 = -Q[w_r] \quad \mbox{in} \ B_4 \setminus B_{1/2}.
\end{equation*}
Using \EQ{mrandf}, \EQ{mrtrap2}, and (f4), we see that for sufficiently large $r> 2r_0$,
\begin{align*}
\inf_{A_r}\left( r^{p} f(u_r,rx) \right) & \geq \frac{1}{2^p} \inf \left\{ |y|^p f(s,y) : r \leq |y| \leq 2r, \ m(r) \leq s \leq \bar C m(r) \right\} \\
& \geq c {m(r)}^{p-1} \inf\left\{ s^{1-p} |y|^p f(s,y) : \mu \leq s \leq C\widetilde\Phi(y), \ r \leq |y| \right\} \\
&  \geq c {m(r)}^{p-1}.
\end{align*}
In particular, for such $r>2r_0$,
\begin{equation*}
-Q[v_r] \geq c{m(r)}^{p-1} \chi_{A_r}\geq 0= -Q[w_r] \quad \mbox{in} \ B_{4} \setminus B_{1/2}.
\end{equation*}
Applying (H3) and (H5), we find that
\begin{equation*}
v_r \geq c_1 m(r) +w_r \quad \mbox{on} \ B_2\setminus B_{1},
\end{equation*}
for some $c_1> 0$ which does not depend on $r$. Using the definition of $v_r$ and $w_r$, together with \re{mrtrap2}, we discover that for sufficiently large $r> 2r_0$,
\begin{equation*}
u(x) - \omega(2r) \left( \widetilde \Phi - 1 \right) \geq c_1 m(r) \geq c_2 \widetilde\Phi \quad \mbox{on} \ B_{2r}\setminus B_r,
\end{equation*}
and therefore
\begin{equation*}
u(x) - \left( \omega(2r) + c_2 \right) \left( \widetilde\Phi - 1 \right) > 0 \quad \mbox{on} \ B_{2r}\setminus B_r,
\end{equation*}
for some $c_2> 0$ which does not depend on $r$. It follows that $\omega(r) \geq \omega(2r) + c_2$ for all sufficiently large $r> 2r_0$, and hence $\omega(r) \rightarrow -\infty$ as $r \to \infty$. This is an obvious contradiction, since $\omega> 0$. Our proof is complete.

\section{Cone-like domains} \label{cones}

In this section we adapt and apply our method to obtain nonexistence results for the semilinear equation
\begin{equation} \label{modelgcone}
-\Delta u = |x|^{-\gamma} g(u)\qquad\mbox{in }\; \mathcal{C}\setminus B_{r_0},
\end{equation}
 where $\mathcal{C}$ is a cone-like domain. Our Theorem \ref{slthm-cones} (see also  Remark \ref{remcon} below) generalizes the previous results on this problem, in which only the case $g(u)=u^\sigma$ was studied.

Our technique for proving the nonexistence of supersolutions of 
$-Q[u] = f(u,x)$
in an unbounded domain $\Omega$ requires the availability of a positive subsolution $\Psi$ of $-Q[\Psi] = 0$ in $\Omega$, which we ``slide underneath" $u$. 
If $\Omega$ is an exterior domain, then $\partial \Omega$ is compact, and so we typically take $\Psi=\Phi$ where $\Phi$ is a fundamental solution of $-Q[\Phi] =0$. We need not worry about the boundary of $\Omega$ when using the comparison principle, since by considering a slightly smaller subdomain, we may assume that $\inf u > 0$ on $\partial \Omega$. In the case that $\Omega$ is a more general cone-like domain, the situation is different. Finding a suitable subsolution $\Psi$ is more involved in this case because our argument requires $\Psi$ to vanish on $\partial \Omega$, which is unbounded.

It is well-known how to construct  $\Psi$ for the Laplacian on a cone, as we recall below. Building such functions $\Psi$ for more general nonlinear operators is an open problem, which we intend to study in the future.

\medskip

Throughout this section we assume that $n\geq 2$ and denote by $\sph$ the unit sphere in $\R^n$, and take $\omega \subseteq \sph$ to be a nonempty, proper, connected, smooth, and relatively open subset of $\sph$. Our cone-like domain is $\mathcal C_\omega$, given by
\begin{equation*}
\mathcal C_\omega : = \left\{ x \in \R^n \setminus \{ 0 \} : \frac{x}{|x|} \in \omega \right\}.
\end{equation*}
Let $\lambda_{1,\omega} $ and $\varphi_{1,\omega}$ denote, respectively, the principal eigenvalue and eigenfunction of the Dirichlet Laplace-Beltrami operator $-\Delta_\theta$ on $\omega$, where $\theta = x/|x|$. Then for $\beta \in \R$, the function
\begin{equation*}
\Psi(x) : = |x|^{-\beta} \varphi_{1,\omega} (\theta)
\end{equation*}
is positive in $\mathcal C_{\omega}$, continuous on $\overline C_\omega \!\setminus \!\{ 0 \}$, smooth in $C_\omega$, and vanishes on $\partial C_\omega \!\setminus \{ 0 \}$. Using the formula
$
\Delta u = \frac{1}{r^2} \Delta_{\theta} u + \frac{n-1}{r} \frac{\partial u}{\partial r} + \frac{\partial^2 u}{\partial r^2},
$
an easy calculation verifies that
\begin{equation*}
-\Delta \Psi(x) = \left( \lambda_{1,\omega} - \beta(\beta+2-n) \right) |x|^{-\beta-2} \varphi_{1,\omega}(\theta).
\end{equation*}
The solutions of the quadratic equation $\beta(\beta+2-n) = \lambda_{1,\omega}$ are\begin{equation*}
\beta^\pm := \frac{1}{2}(n-2) \pm \frac{1}{2} \sqrt{ (n-2)^2 + 4 \lambda_{1,\omega}}.
\end{equation*}
Therefore the functions defined by
\begin{equation*}
\Psi^\pm(x): = |x|^{-\beta^\pm} \varphi_{1,\omega}(\theta)
\end{equation*}
are positive and harmonic in $\mathcal{C}_\omega$, 
and $\Psi^\pm = 0$ on $\partial C_\omega \setminus \{ 0 \}$.
Notice that since $\omega$ is a proper subset of $\sph$, the eigenvalue $\lambda_{1,\omega} $ is strictly positive and thus we have
\begin{equation}\label{betabnd}
\beta^- < 0 \leq n-2 < \beta^+.
\end{equation}
Hence  $\Psi^+(x) \to 0$ as $|x|\to \infty$, while $\Psi^-(x)$ is unbounded for large $|x|$. 
\medskip

Let us now state our nonexistence results for supersolutions of the semilinear equation \eqref{modelgcone} in cone-like domains. We define the constants
\begin{equation}
\sigma^\pm : = 1 + \frac{2-\gamma}{\beta^\pm}.
\end{equation}
Note that $\beta^\pm\neq 0$ due to \eqref{betabnd},  so $\sigma^\pm$ are well-defined and $\sigma^- < 1 < \sigma^+$. In addition $\sigma^+ < 1 + \frac{2-\gamma}{n-2}$ if $n\geq 3$.

\begin{thm} \label{slthm-cones}
Suppose that $\gamma < 2$, 
$g:(0,\infty) \to (0,\infty)$ is continuous and \begin{equation} \label{local-cones}
\liminf_{s\searrow 0} s^{-\sigma^+} g(s) > 0, \quad \mbox{and} \quad \liminf_{s\to \infty} s^{-\sigma^-} g(s) > 0.
\end{equation} Then the differential inequality
\begin{equation}\label{cone-eq}
-\Delta u \geq |x|^{-\gamma} g(u) \quad \mbox{in} \ \mathcal C_\omega \setminus B_{r_0}
\end{equation}
does not possess a positive solution $u> 0$, for any $r_0> 0$.
\end{thm} 
 \begin{remark}\label{remcon}
Prior to this work the inequality 
 \begin{equation}\label{cone-power-eq}
-\Delta u \geq |x|^{-\gamma} u^\sigma
\end{equation}
has been extensively studied in cone-like domains.  Bandle and Levine \cite{BaL}, Berestycki, Capuzzo-Dolcetta, and Nirenberg \cite{BCN} showed that \re{cone-power-eq} does not have positive solutions in $C_\omega$ if $\sigma\in (1,\sigma^+)$. More recently Kondratiev, Liskevich, Moroz, and Sobol \cite{KLM,KLMS}, discovered the critical exponent $\sigma^-$ in the sublinear case $\sigma<1$, and proved that \re{cone-power-eq} does not have positive solutions in $\mathcal C_\omega \setminus B_{r_0}$ if and only if $\sigma\in [\sigma^-,\sigma^+]$.
Theorem \ref{slthm-cones} reveals the role of these two exponents with respect to a general nonlinearity $g$: the exponent $\sigma^-$ concerns only its behavior at infinity, and $\sigma^+$ the behavior near zero. 
\end{remark}

It what follows, it will be useful to denote the sets
\begin{equation} \label{setE}
E(\omega,r,R ): = \mathcal{C}_\omega \cap \left(B_{R} \setminus B_r\right), \qquad E(\omega,r): = E(\omega,r,2r),
\end{equation}
and, given a function $u$ on $\mathcal C_\omega$, the quantities
\begin{equation} \label{rho-cones}
\rho^\pm(\omega,r,R) : = \inf_{E(\omega,r,R)} \frac{u}{\Psi^\pm}\,, \qquad \rho^\pm(\omega, r) : = \rho^\pm(\omega, r,2r).
\end{equation}

Throughout this section, it will be convenient for us to interpret the statement $
u \geq v $ on $\partial \Omega$,
for functions $u$ and $v$ possibly defined only on the domain $\Omega$, to mean
\begin{equation*}
\liminf_{x\to \partial \Omega} \, (u-v)(x) \geq 0.
\end{equation*}

\begin{remark} \label{hopfrem}
We recall that by Hopf's lemma, if $w$ is positive and superharmonic in a smooth bounded domain $\Omega$, then $w \geq c \dist(x,\partial \Omega)$ for some small $c>0$. Hence any function $u$ which is positive and superharmonic  in $C_\omega\setminus B_{r_0}$ satisfies $u \geq c \dist(x,\partial C_\omega)\ge c\Psi^\pm$ in $E(\omega,r,R)$, for any $0 < r_0<r <R$, where the constant $c> 0$ depends only on $u$, $r_0$, $r$, and $R$.
\end{remark}

With our functions $\Psi^\pm$ in hand, the generalization of Lemma \ref{fundycmpl} to cone-like domains is relatively straightforward.

\begin{lem} \label{fundycmpl-cone}
Suppose that $u> 0$ is superharmonic in $\mathcal C_\omega \setminus B_{r_0}$. Then there exist constants $c$ and $C$ which depend on $u$ and $r_0$, but not on $r$, such that
\begin{equation} \label{fundycmp-cones}
0< c \leq \rho^+(\omega,r) \quad \mbox{and} \quad 0 < \rho^-(\omega,r) \leq C \quad \mbox{for every} \ r \geq 2r_0,
\end{equation}
where $\rho^\pm(\omega,r)$ are given by \eqref{rho-cones}. Moreover, $r\mapsto \rho^+(\omega,r)$ is nondecreasing.
\end{lem}
\begin{proof}
Using Remark \ref{hopfrem}, $\rho^+(\omega,r)>0$ for any $r>r_0$. Since $\Psi^+ = 0$ on $\partial \mathcal C_\omega \setminus \{ 0 \}$ and $\Psi^+ (x) \rightarrow 0$ as $|x| \to \infty$, for each $\ep > 0$ we may select $R> 2r$ so large that
\begin{equation*}
u + \epsilon \geq \rho^+(\omega,r)\Psi^+  \quad \mbox{in } \ E(\omega,r)\quad\mbox{and }\mbox{ in }\; (\rn\setminus B_{R})\cap\mathcal{C}_\omega.
\end{equation*}
By applying the weak comparison principle and sending $R\to \infty$ and then $\ep \to 0$, we deduce that  the map $r\mapsto \rho^+(\omega,r)$ is a nondecreasing function, from which the first part of \re{fundycmp-cones} follows, with $c=\rho^+(\omega,2r_0)$.

Next we show that $\rho^-(\omega,r)$ is bounded above in $r$. Defining
\begin{equation*}
\hat\Psi^- (x) : = \Psi^-(x) - \sup_{E(\omega,r_0)} \Psi^-,\qquad \hat \rho^-(\omega,r) : = \inf_{E(\omega,r)} \frac{u}{\max\{ 0, \hat\Psi^-\}},
\end{equation*}
we first claim that
the map $ r \mapsto \hat \rho^-(\omega,r)$ is nonincreasing on $ (r_0,\infty)$.
Since $\beta_\omega^- < 0$, we have that $\sup_{E(\omega,r)} \Psi^- > \sup_{E(\omega,r_0)} \Psi^-$ for every $r> r_0$. Hence for each $r>r_0$ the quantity $\hat\rho^-(\omega,r)$ is finite and positive, and 
$u \geq \hat\rho^-(\omega,r) \hat\Psi^-$ on $\partial E(\omega,2r_0,2r)$.
Hence the weak comparison principle implies
\begin{equation*}
u \geq \hat\rho^-(\omega,r) \hat\Psi^- \quad \mbox{in} \  E(\omega,2r_0,2r)
\end{equation*}
for every $r> r_0$, from which the claim follows. Since $\rho(\omega,r) \leq \hat\rho(\omega,r)$, the bound $\rho(\omega,r) \leq C$ follows at once.
\end{proof}

Our proof of Theorem \ref{slthm-cones} requires a technical lemma, which is convenient for handling difficulties which arise due to the boundary of the cone $\mathcal C_\omega$. We prove it now, before proceeding to the proof of the theorem.

\begin{lem}\label{nastybd-cones}
Let $\omega$ be as in Theorem \ref{slthm-cones}. Given $b > 0$ and $\omega' \subset \subset \omega$, there exists $\ep = \ep(\omega,\omega',b) > 0$ such that if $u$ satisfies
\begin{equation*}
\left\{ \begin{aligned}
& -\Delta u \geq 0 & \mbox{in} & \ E(\omega, \textstyle\frac{1}{2},4), \\
& u \geq - \ep & \mbox{in} & \ E(\omega, \textstyle\frac{1}{2},4), \\
& u \geq 0 & \mbox{on} & \ \partial \mathcal{C}_\omega\cap( B_4\setminus B_{1/2}),\\
& u \geq b & \mbox{in} & \ E(\omega',1),
\end{aligned}
\right.
\end{equation*}
then $u \geq 0$ in $E(\omega,1)$.
\end{lem}
\begin{proof}
We denote the domain $
\Omega : = E(\omega,\textstyle\frac{1}{2},4)  \setminus E(\omega', 1 )$.
Let $v_1$ and $v_2$ be the solutions of the Dirichlet problems
\begin{equation*}
\left\{ \begin{aligned}
& -\Delta v_i = 0 & \mbox{in} & \ \Omega,\\
& v_i = g_i & \mbox{on} & \ \partial \Omega,
\end{aligned} \right.
\end{equation*}
where $g_1 = g_2 = 0$ on the sides $\partial \mathcal C_\omega \cap (B_4\setminus B_{1/2})$ of the outer part of the boundary of $\Omega$, $g_1= b>0$ and $g_2 = 0$ on the inner boundary $\partial E(\omega',1)$, and finally $g_1=0$ and $g_2 = 1$ on the top and bottom parts $\mathcal C_\omega \cap \partial\! \left( B_4 \cap B_{1/2} \right)$ of the outer boundary. Elliptic estimates and Hopf's lemma obviously imply that  if $\ep > 0$ is sufficiently small then
\begin{equation*}
v_1 > \ep v_2 \quad \mbox{in the set} \ \Omega \cap \left( B_2 \setminus B_1\right) = \left( \mathcal{C}_\omega \setminus \mathcal C_{\omega'} \right) \cap \left( B_2 \setminus B_1 \right).
\end{equation*}
Set $v: = v_1 - \ep v_2$, and observe that $u \geq v$ on $\partial \Omega$. Therefore by the comparison principle, $u \geq v$ in $\Omega$. In particular, $u > 0$ in $\left( \mathcal{C}_\omega \setminus \mathcal C_{\omega'} \right) \cap \left( B_2 \setminus B_1 \right)$. Since we have $u \geq b > 0$ in $E(\omega',1) = \mathcal C_{\omega'} \cap \left( B_2 \setminus B_1 \right)$ by hypothesis, we obtain $u\geq 0$ in $E(\omega,1)$, as desired.
\end{proof}

\begin{proof}[{Proof of Theorem \ref{slthm-cones}}]
Suppose for contradiction that $u> 0$ satisfies \eqref{cone-eq}. Let us rescale, setting $u_r(x) : = u(rx)$ and observe that $u_r$ satisfies
\begin{equation*}
-\Delta u_r \geq r^{2-\gamma} |x|^{-\gamma} g(u_r) \quad \mbox{in} \ \mathcal C_\omega \setminus B_{r_0/r}.
\end{equation*}
In particular, for $r\geq 2r_0$ we have
\begin{equation} \label{rescale-cone}
-\Delta u_r \geq c r^{2-\gamma} g(u_r) \quad \mbox{in} \ E\!\left(\omega, \textstyle\frac{1}{2},4\right).
\end{equation}
Select a subdomain $\omega' \subseteq \sph$ so that $\overline\omega ' \subseteq \omega$, and define the quantity
\begin{equation*}
m(r) : =  \inf_{E(\omega',1)} u_r = \inf_{E(\omega',r)} u > 0,
\end{equation*}
for  $r >2 r_0$. Next we define
\begin{equation*}
A_r : = \left\{ x\in E(\omega',1) : m(r) \leq u_r \leq \bar C m(r) \right\},
\end{equation*}
where as before $\bar C > 1$ is a fixed constant large enough that the weak Harnack inequality implies $|A_r| \geq c > 0$ for some constant $c> 0$ which depends on $u$ but not on $r>2r_0$. The quantitative strong maximum principle (Lemma \ref{kslap}) then implies that
\begin{equation} \label{noloit-cones}
m(r) \geq c r^{2-\gamma} \min_{m(r) \leq s \leq \bar Cm(r)} g(s).
\end{equation}
Hence either $m(r) \to 0$ or $m(r) \to \infty$ as $r\to \infty$.

\medskip

Suppose first that $\lim_{r\to \infty} m(r) = 0$. Then the estimate \eqref{noloit-cones} and hypothesis \EQ{local-cones} imply that $m(r) \geq cr^{2-\gamma} (m(r))^{\sigma^+}$ for sufficiently large $r$. Since $\beta^+ = \frac{2-\gamma}{\sigma^+ -1}$, we deduce that $m(r) \leq C r^{-\beta^+}$. On the other hand, since $
\inf_{E(\omega',r)} \Psi^+ \geq cr^{-\beta^+},
$
 Lemma \ref{fundycmpl-cone} implies
\begin{equation} \label{matchlowbdcones}
m(r) \geq \rho^+(\omega,r) \inf_{E(\omega',r)} \Psi^+ \geq c r^{-\beta^+}.
\end{equation}
Hence for sufficiently large $r$ we have the two-sided bound
\begin{equation} \label{twosided-cone}
cr^{-\beta^+} \leq m(r) \leq C r^{-\beta^+}.
\end{equation}
To obtain a contradiction, we will show that $\rho^+(\omega,r) \to \infty$ as $r\to \infty$; considering the homogeneity of $\Psi^+$, this will contradict the upper bound in \eqref{twosided-cone}. Recalling \eqref{rescale-cone}, the lower bound in \eqref{twosided-cone} and the hypothesis \EQ{local-cones} imply that
\begin{equation*}
-\Delta u_r \geq c r^{2-\gamma} r^{-\sigma^+ \beta^+} \chi_{A_r} = c r^{-\beta^+} \chi_{A_r} \quad \mbox{in} \ E\!\left(\omega, \textstyle\frac{1}{2},4\right)
\end{equation*}
for large enough $r$. Here we have used the fact that $2-\gamma = (\sigma^+-1) \beta^+$. Since the map $r\mapsto \rho^+(\omega,r)$ is nondecreasing, we have 
\begin{equation*}
v_r(x) : = u_r(x) - \rho^+(\omega,r/2) \Psi^+(rx) \geq 0 \quad \mbox{and }\; -\Delta v_r\ge c r^{2-\gamma} g(u_r) \quad \mbox{in} \ \mathcal C_{\omega} \setminus B_{1/2}.
\end{equation*}
The quantitative strong maximum principle then yields the estimate
\begin{equation*}
v_r \geq cr^{-\beta^+} \geq a \Psi^+(rx) \quad \mbox{for all} \ x\in E\left(\omega',1\right),
\end{equation*}
where $a>0$ does not depend on $r$. For $0 < \delta < 1$ to be chosen below, the function
\begin{equation*}
w_r(x) : = v_r(x) - \delta a \Psi^+(rx)
\end{equation*}
satisfies $-\Delta w_r \geq 0$ in $E(\omega, \textstyle \frac{1}{2},4)$ and, by the homogeneity of $\Psi^+$,
\begin{equation}\label{manyineqs}
\left\{ \begin{aligned}
& w_r \geq -C \delta a r^{-\beta^+} & \mbox{in} & \ E\left(\omega, \textstyle\frac{1}{2},4 \right), \\
& w_r \geq 0 & \mbox{on} & \ \partial \mathcal{C}_\omega\cap( B_4\setminus B_{1/2}), \\
& w_r \geq ca (1-\delta)r^{-\beta^+} & \mbox{in} & \ E\left(\omega',1\right).
\end{aligned} \right.
\end{equation}
Hence if $\delta > 0$ is chosen sufficiently small (depending on $\omega$, $\omega'$ and $\Psi^+$ but not on $r$), Lemma \ref{nastybd-cones} implies that $w_r \geq 0$ in $E(\omega,1)$. Unwinding the scaling, we deduce that $\rho^+(\omega,r) \geq \rho^+(\omega,r/2) + \delta a$. We therefore have our contradiction $\rho^+(\omega,r) \to \infty$ as $r\to\infty$, as desired. This completes the proof in the case that $m(r) \to 0$ as $r\to \infty$.

\medskip

We have left to consider the  alternative  $m(r) \to \infty$ as $r \to \infty$. In this case,  \eqref{local-cones} and \eqref{noloit-cones} give us the estimate $m(r) \geq cr^{2-\gamma} \left( m(r) \right)^{\sigma^-}$ for large $r>r_0$, from which we obtain the lower bound $
m(r) \geq c r^{-\beta^-}$.
We next show that in fact we have the matching upper bound
\begin{equation} \label{matchuppbdcones}
m(r) \leq Cr^{-\beta^-}
\end{equation}
for large $r$. In comparison with previous arguments in this paper, extra care is needed in the proof of \EQ{matchuppbdcones}, since $\omega \neq \omega'$. In particular, we must concern ourselves with the possibility that  $\rho^-(\omega',r)$ is much larger than $\rho^-(\omega,r)$. We rule out this possibility using Lemma \ref{nastybd-cones}.

To get \eqref{matchuppbdcones}, we first define the function
\begin{equation*}
 w_r(x): = u(rx) - \delta \rho^-(\omega',r) \Psi^-(rx),
\end{equation*}
for each $r> 2r_0$, and $0 < \delta < 1$ to be chosen. Notice that $w$ satisfies 
\begin{equation*}
\left\{ \begin{aligned}
& w_r \geq -C \delta  \rho^-(\omega',r)  r^{-\beta^-} & \mbox{in} & \ E\left(\omega, \textstyle\frac{1}{2},4\right), \\
& w_r \geq 0 & \mbox{on} & \ \partial \mathcal{C}_\omega\cap( B_4\setminus B_{1/2}), \\
& w_r \geq c (1-\delta)  \rho^-(\omega',r)  r^{-\beta^-} & \mbox{in} & \ E\left(\omega',1\right).
\end{aligned} \right.
\end{equation*}
Since $-\Delta w_r \geq 0$ in $E(\omega,\textstyle\frac{1}{2},4)$, according to Lemma \ref{nastybd-cones} we can fix $\delta > 0$ small enough that $w_r \geq 0$ in $E(\omega,1)$. Therefore
\begin{equation*}
u_r(x) \geq \delta \rho^-(\omega',r) \Psi^-(rx)  \quad \mbox{for every} \ x \in E(\omega,1),
\end{equation*}
from which it follows that
\begin{equation*}
\rho^-(\omega, r) \geq \delta \rho^-(\omega', r).
\end{equation*}
Recalling that $\rho^-(\omega,r) \leq C$ by Lemma \ref{nastybd-cones}, we thereby obtain
\begin{equation*}
m(r) = \inf_{E(\omega',r)} u \leq \rho^-(\omega',r)  \displaystyle\sup_{E(\omega',r)} \Psi^- \leq C r^{-\beta^-},
\end{equation*}
where we have also used the homogeneity of the function $\Psi^-$. We have proved the estimate \EQ{matchuppbdcones}, and so we now have the two-sided bound
\begin{equation}\label{hardbndcones}
cr^{-\beta^-} \leq m(r) \leq Cr^{-\beta^-} \quad \mbox{for sufficiently large} \ r > r_0.
\end{equation}

With $\hat\rho$ and $\hat\Psi^-$ as in the proof of Lemma \ref{fundycmpl-cone}, let us define the function
\begin{equation*}
v_r(x) : = u_r(x) - \hat\rho(\omega, 2r) \hat\Psi^-(rx).
\end{equation*}
As we showed in the proof of Lemma \ref{fundycmpl-cone}, we have $v_r \geq 0$ in $B_4\setminus B_{1/2}$ for all $r > 2r_0$, and $v_r$ satisfies
\begin{equation*}
-\Delta v_r \geq c r^{2-\gamma} (m(r))^{-\sigma^-}\chi_{A_r} \geq cr^{-\beta^-} \chi_{A_r} \quad \mbox{in} \ E\left(\omega, \textstyle\frac{1}{2},4\right).
\end{equation*}
Here we have used \eqref{local-cones}, \eqref{rescale-cone}, \eqref{hardbndcones} and the fact that $2-\gamma - \sigma^-\beta^- = -\beta^-$. The quantitative strong maximum principle then implies that
\begin{equation*}
v_r(x) \geq cr^{-\beta^-} \geq a \Psi^-(rx) \quad \mbox{for all} \ x\in E\left(\omega', 1 \right),
\end{equation*}
for some $a> 0$ independent of $r$. For $0 < \delta < 1$ to be selected, the function $\tilde v_r : = v_r - \delta a \Psi^-(rx)$ satisfies \re{manyineqs} with $\beta^+$ replaced by $\beta^-$.
Hence if $\delta > 0$ is small enough, Lemma \ref{nastybd-cones} implies that $\tilde v_r \geq 0$ in $E(\omega,1)$. In particular, we have
\begin{equation*}
u_r(x) \geq \hat\rho(\omega,2r) \hat\Psi^-(rx) +\delta a \Psi^- (rx) \geq \left( \hat\rho(\omega,2r) + \delta a\right) \hat\Psi^-(rx)
\end{equation*}
for every $x\in E(\omega,1)$. It follows that $\hat\rho(\omega, r) \geq \hat\rho(\omega,2r) + \delta a$. This implies the absurdity $\hat\rho(\omega,r) \to -\infty$ as $r\to \infty$, completing the proof of the theorem.
\end{proof}

Theorem \ref{slthm-cones} is easily seen to be sharp by ``bending" the functions $\Psi^\pm$, that is, considering $(\Psi^\pm)^\tau$ for some appropriate $0 < \tau < 1$. One may also consult for example \cite{KLS}.

\medskip

We conclude this section by stating the previous result in the particular case of the half space
\begin{equation*}
\R^n_+ : = \left\{ (x',x_n) \in \R^{n-1}\times\R :  x_n > 0 \right\}.
\end{equation*}
We have $\R^n_+ = \mathcal C_\omega$, where $\omega$ is the upper hemisphere. It is simple to see that $\Phi^-(x) = x_n$, and $\Phi^+$ is its Kelvin transform given by $\Phi^+(x) = |x|^{-n} x_n$. In particular, $\beta^- = -1$ and $\beta^+ = n-1$. We thereby obtain:

\begin{cor}
Assume that $n\geq 2$, $\gamma < 2$, $r_0>0$, $g:(0,\infty) \to (0,\infty)$ is continuous and satisfies
\begin{equation*}
\liminf_{s\searrow 0} s^{-(n+1-\gamma)/(n-1)} g(s) > 0 \qquad \mbox{and} \qquad \liminf_{s\to \infty} s^{1-\gamma} g(s) > 0.
\end{equation*}
Then there does not exist a positive supersolution of the equation
\begin{equation*}
-\Delta u = |x|^{-\gamma} g(u) \quad \mbox{in} \ \R^n_+ \setminus B_{r_0}.
\end{equation*}
\end{cor}

\section{Systems of inequalities} \label{systems}
\subsection{Systems of elliptic inequalities}
The method we developed in the previous section  generalizes easily to systems of the form
\begin{equation}\label{system}
-Q_i [u_i] \ge f_i(u_1,\ldots,u_N,x),\quad i=1,\ldots,N,
\end{equation}
where $f$ is a positive, continuous function on $\left( \R^n\setminus B_{r_0} \right) \times (0,\infty)^N$. Our approach essentially reduces the question of existence of positive solutions of \re{system} to that of systems of certain algebraic inequalities.

Assuming that $Q_i$ satisfies (H1)-(H5) with constants $p_i$, $\alpha^*_i$, $\widetilde\alpha^*_i$, and so on, we may use the Harnack inequality and the quantitative strong maximum principle as before to obtain
\begin{equation}\label{parteq2}
 m_i^{p_i-1}(r)\ge  c r^{p_i} \inf_{(s_1,\ldots,s_N,x)\in A_r} f_i(s_1,\ldots, s_N,x) \quad \mbox{for every} \ 1 \leq i \leq N,
\end{equation} where we have set $m_i(r) := \inf_{r\le |x|\le 2r} u_i(x)$ as well as
\begin{equation*}
A_r:= \left\{ (s_1,\ldots,s_N,x) : r\le |x|\le 2r,\; m_i(r)\le s_i\le \bar Cm_i(r)\ \mbox{for} \ i=1,\ldots,N \right\}.
\end{equation*}
The game is then to impose hypotheses on the functions $f_i$ which ensure that the inequalities \EQ{parteq2} are incompatible with those of Lemma \ref{mbounds_1} for large~$r$. In the ``critical" cases (like for $-\Delta u \ge u^\sigma$ with $\sigma = n/(n-2)$), we typically obtain a two-sided bound on some $m_j(r)$ for some $j$ and large $r$, rescale the function $u_j$, and then proceed as in the proof of Theorem~\ref{FNL} to obtain a contradiction.

For instance, if we consider the system
\begin{equation}\label{qij1}
\left\{ \begin{aligned}
 -Q_1 [u] & \geq |x|^{-\gamma}u^{a}v^{b}, \\
- Q_2 [v] & \geq |x|^{-\delta}u^{c}v^{d},
\end{aligned} \right.
\end{equation}
in some exterior domain of $\R^n$, then \re{parteq2} becomes
\begin{equation}\label{qij2}
\left\{ \begin{aligned}
m_1^{p_1-1}(r) & \geq c r^{p_1-\gamma} m_1^{a}(r) m_2^{b}(r), \\
m_2^{p_2-1}(r) & \geq c r^{p_2-\delta} m_1^{c}(r) m_2^{d}(r),
\end{aligned} \right.
\end{equation}
for sufficiently large $r$. We may then combine \re{qij2} with the inequalities given by Lemma~\ref{mbounds_1} in order to determine the set of parameters for which these inequalities are incompatible. Notice that we may take parameters $a,b,c,d\in \mathbb{R}$, and in particular we can consider systems with various singularities. It is also possible to consider operators $Q_1$ and $Q_2$ which are of a different nature; e.g., $Q_1$ may be the $p$-Laplacian while $Q_2$ is an Isaacs operator.

\medskip

Naturally, any attempt at stating a very general result for a system of the form \EQ{system} is immediately met with a combinatorial explosion of cases to consider (e.g. various signs of $\alpha_i^*$ and $\widetilde\alpha_i^*$, corresponding requirements on the functions $f_i$ as some $s_j$ are going to zero while others are at infinity, etc). While it will be apparent that our techniques are sufficiently flexible to yield nonexistence results for such general systems, with an eye toward the clarity of our presentation we study here only some special cases, which however illustrate the general approach well enough.

\subsection{The extended Lame-Emden system} Let us calculate the set of parameters $(\sigma_1,\sigma_2) \in \R^2$ for which the system
\begin{equation}\label{LE}
\left\{ \begin{aligned}
-Q_1 [u] & \geq v^{\sigma_1}, \\
-Q_2 [v] & \geq u^{\sigma_2},
\end{aligned} \right.
\end{equation}
has no positive solutions in any exterior domain of $\R^n$. We first consider the case  $\sigma_1,\sigma_2 \geq~0$. Then we obtain from \EQ{qij2} the bounds
\begin{equation}\label{LE1}
\left\{ \begin{aligned}
& {(m_1(r))}^{\sigma_1\sigma_2-(p_1-1)(p_2-1)}\le C r^{-(p_1(p_2-1)+\sigma_1p_2)} , \\
& {(m_2(r))}^{\sigma_1\sigma_2-(p_1-1)(p_2-1)}\le C  r^{-(p_2(p_1-1)+\sigma_2p_1)}.
\end{aligned} \right.
\end{equation}
If $\sigma_1\sigma_2\leq (p_1-1)(p_2-1)$, then sending $r\to\infty$ in \EQ{LE1} immediately yields a contradiction with the second set of inequalities in Lemma~\ref{mbounds_1}  (recall that $\widetilde \alpha^*> -1$ and $p_i > 1$).
If $\sigma_1\sigma_2>(p_1-1)(p_2-1)$ then \EQ{LE1} and Lemma ~\ref{mbounds_1} yield
\begin{equation*}
cr^{-\max\{\alpha_1^*,0\}}\le m_1(r)\le C r^{-\frac{p_1(p_2-1)+\sigma_1p_2}{\sigma_1\sigma_2-(p_1-1)(p_2-1)}},
\end{equation*}
and the same inequality for $m_2(r)$, with permuted indices. This is of course a contradiction, if
\begin{equation*}
\alpha_1^*<\frac{p_1(p_2-1)+\sigma_1p_2}{\sigma_1\sigma_2-(p_1-1)(p_2-1)}\qquad\mbox{or}
\qquad
\alpha_2^*<\frac{p_2(p_1-1)+\sigma_2p_1}{\sigma_1\sigma_2-(p_1-1)(p_2-1)}\,.
\end{equation*}
If neither of these strict inequalities holds, but equality holds say in the first, then the rescaled functions
$
u_r = r^{\alpha_1^*}u(rx)$, $v_r = r^{\alpha_2^*}v(rx)$
satisfy the system
\begin{equation}\label{LE2}
\left\{ \begin{aligned}
& -Q_1 [u_r] \geq r^{\alpha_1^*(p_1-1)+p_1-\sigma_1\alpha_2^*}  v_r^{\sigma_1}, \\
& -Q_2 [v_r] \geq r^{\alpha_2^*(p_2-1)+p_2-\sigma_2\alpha_1^*}  u_r^{\sigma_2},
\end{aligned} \right.
\end{equation}
in  $\rn\setminus B_{1/2}$, for sufficiently large $r$. Moreover, (H4) implies that we have $ c \leq u_r \leq C$ on a subset of $B_2\setminus B_1$ with measure bounded below by a positive constant which does not depend on $r$. By applying the quantitative strong maximum principle to the second, then to the first equation in this system, and using the equality $\displaystyle\alpha^*_1(p_1-1) + p_1 -\sigma_1\alpha^*_2 + \sigma_1\frac{\alpha^*_2(p_2-1)+p_2-\sigma_2\alpha^*_1}{p_2-1}=0$, we obtain
\begin{equation}\label{tukk}
-Q_1 [u_r] \ge c  >0
\end{equation}
on a subset of $B_2\setminus B_1$ which has measure bounded below by a positive constant independent of $r$. We now proceed as in the proof of Theorem \ref{FNL} to deduce from \EQ{tukk} that $r^{\alpha_1^*} m_1(r) \to \infty$ as $r\to \infty$, a contradiction. This completes the argument in the case  $\sigma_1,\sigma_2 \geq 0$, and we have found that we have nonexistence of positive solutions provided either $\sigma_1\sigma_2 \leq (p_1-1)(p_2-1)$, or
\begin{equation*} \min\left\{ \alpha_1^*-\frac{p_1(p_2-1)+\sigma_1p_2}{\sigma_1\sigma_2-(p_1-1)(p_2-1)}\,,\;
\alpha_2^*-\frac{p_2(p_1-1)+\sigma_2p_1}{\sigma_1\sigma_2-(p_1-1)(p_2-1)}  \right\} \leq 0.
\end{equation*}

Next, note that if $\sigma_1<0$ and $\widetilde\alpha^*_2\ge 0$, we obtain from \re{qij2} and Lemma \ref{mbounds_1} that
$$
Cr^{p_1-1}\ge m_1^{p_1-1}(r)\ge r^{p_1} m_2^{\sigma_1}(r)\ge r^{p_1}(\log r)^{\sigma_1}$$
for all sufficiently large $r$, which is an obvious contradiction. Similarly, $\sigma_2<0$ and $\widetilde\alpha^*_1\ge 0$ is impossible.

Consider now the case that $\sigma_1 < 0$ and $\sigma_2<0$, together with  $\widetilde \alpha_1^*< 0$ and $\widetilde \alpha_2^*< 0$. Then \re{qij2} and Lemma \ref{mbounds_1} imply
\begin{equation*}
\left\{ \begin{aligned}
 & cr^{-\widetilde \alpha_1^*(p_1-1)}\ge m_1^{p_1-1}(r)\ge cr^{p_1}m_2^{\sigma_1}\ge r^{p_1-\widetilde\alpha_2^*\sigma_1}\\
 & cr^{-\widetilde \alpha_2^*(p_2-1)}\ge m_2^{p_2-1}(r)\ge cr^{p_2}m_1^{\sigma_2}\ge r^{p_2-\widetilde\alpha_1^*\sigma_2}
 \end{aligned}\right.
\end{equation*}
and we get a contradiction provided
$$
p_1+\widetilde \alpha_2^*(-\sigma_1) + \widetilde \alpha_1^*(p_1-1)>0\qquad \mbox{or}\qquad p_2+\widetilde \alpha_1^*(-\sigma_2) + \widetilde \alpha_2^*(p_2-1)>0.
$$ If one of these inequalities is an equality, we rescale $u$ or $v$ as above, replacing $\alpha_i$ by $\widetilde\alpha_i$, and $\Phi_i$ by $\widetilde \Phi_i$, to reach a contradiction in the same manner.

\medskip

Finally, in the case that $\sigma_1<0$,  $\widetilde \alpha_2^*< 0$, $\sigma_2\ge 0$, we get
\begin{equation*}
\left\{ \begin{aligned}
 & cr^{-\widetilde \alpha_1^*(p_1-1)}\ge m_1^{p_1-1}(r)\ge cr^{p_1}m_2^{\sigma_1}\ge r^{p_1-\widetilde\alpha_2^*\sigma_1}\\
 & cr^{-\widetilde \alpha_2^*}\ge m_2(r)\ge r^{\frac{p_2(p_1-1)+\sigma_2p_1}{(p_1-1)(p_2-1)-\sigma_1\sigma_2}}
 \end{aligned}\right.
\end{equation*}
from which the arguments above give a contradiction provided that
\begin{equation*}
p_1+\widetilde \alpha_2^*(-\sigma_1) + \widetilde \alpha_1^*(p_1-1)\ge0\qquad \mbox{or}\qquad \frac{p_2(p_1-1)+\sigma_2p_1}{(p_1-1)(p_2-1)-\sigma_1\sigma_2}\ge -\widetilde \alpha_2^*.
\end{equation*}

We leave it to the reader to check that if the nonexistence hypotheses above are not satisfied, then the system \re{LE} has positive solutions in exterior domains, which can be easily constructed with the help of the fundamental solutions $\Phi_i$, $\widetilde \Phi_i$.

\subsection{Autonomous systems of three or more inequalities}
Let us now consider a system of the form
\begin{equation} \label{LEsys}
\left\{ \begin{aligned}
 -Q_1\left[ u_1 \right] &\geq f_k(u_k), \\
 -Q_2\left[ u_2 \right] &\geq f_1(u_1), \\
& \ldots \\
 -Q_k \left[ u_{k} \right] &\geq f_{k-1}(u_{k-1}).
\end{aligned} \right.
\end{equation}
For simplicity we assume that
\begin{equation}\label{noh4}
p_i = p >1\quad\mbox{ and }\quad \widetilde \alpha^*_i > 0,  \qquad \mbox{for every} \ i = 1,\ldots, k.
\end{equation}
The latter hypothesis renders it unnecessary to form an analogue of condition (f4). Let us state the assumptions on the functions $f_i$ which will ensure nonexistence of positive solutions of \EQ{LEsys}. We suppose that the nonlinearity $ f_i:(0,\infty) \to (0,\infty)$ is continuous for each $ i = 1,\ldots, k$,
as well as
\begin{equation} \label{fsysloc}
0 < \liminf_{s\searrow 0} s^{-\sigma_i} f_i(s) \leq \infty \quad \mbox{for each} \ i = 1,\ldots, k,
\end{equation}
for some exponents $\sigma_1,\ldots, \sigma_k > 0$. We denote
\begin{equation}\label{D}
D : = \prod_{i=1}^k \sigma_i - (p-1)^k
\end{equation}
and assume for the moment that $D> 0$  (we will see later that we may assume without loss of generality that the geometric mean of $\sigma_1,\ldots, \sigma_k$ is at least $p-1$). For each $1 \leq i \leq k$, define the constant
\begin{equation}\label{betak}
\beta_i := \frac{p}{D} \sum_{j=0}^{k-1} \left( (p-1)^j \prod_{l=0}^{k-2-j} \sigma_{k+i-1-l}\right),
\end{equation}
where for notational convenience for $i > k$ we set  $\sigma_i : =\sigma_{(i \!\mod k)}$, $u_i : =u_{(i \!\mod k)}$, and so on, and we define an empty product to have the value of 1. For example, $\beta_1$ is given by the expression
\begin{multline*}
\beta_1 = \frac{p}{\sigma_1 \sigma_2\cdots \sigma_k - (p-1)^k}\left( (p-1)^{k-1} + (p-1)^{k-2} \sigma_k \right. \\
\left. + (p-1)^{k-3} \sigma_k\sigma_{k-1} + \ldots + (p-1) \sigma_k \cdots \sigma_3 + \sigma_k\cdots \sigma_2 \right)
\end{multline*}
We will argue that the system \EQ{LEsys} has no solution $u_1, u_2, \ldots, u_k > 0$ in any exterior domain of $\R^n$ provided that
\begin{equation*}
\min_{1\leq i \leq k} \left( \alpha_i -\beta_i  \right) \leq 0.
\end{equation*}
It clearly suffices to show that $
\alpha_1 \leq \beta_1 $ implies the nonexistence of a positive solution of \EQ{LEsys}.
Arguing by contradiction, we assume that \EQ{LEsys} has a solution $u_1, u_2, \ldots u_k > 0$ in  some exterior domain $\R^n \setminus B_{r_0}$ but that $\alpha_1\leq \beta_1$. Denote
\begin{equation*}
u_{i,r} (x) := u_i(rx) \quad \mbox{and} \quad m_i(r) := \inf_{B_{2r} \setminus B_r} u_i = \inf_{B_2\setminus B_1} u_{i,r}, \quad r> 2r_0, \ i = 1,\ldots,k.
\end{equation*}
For every $r>2r_0$ and $1\leq i \leq k$,
\begin{equation*}
-Q_{i+1}[ u_{i+1,r} ] \geq r^p f_{i}(u_{i,r}) \quad \mbox{in} \ \R^n \setminus B_{1/2}.
\end{equation*}
Arguing as in the proof of Theorem \ref{FNL}, for $r>2r_0$ and $1\leq i \leq k$ we obtain
\begin{equation}\label{keyineq}
m_{i+1} (r)^{p-1} \geq c r^{p} \inf\left\{ f_i(s) : m_{i}(r) \leq s \leq \bar C m_{i}(r) \right\}, \quad \mbox{for all} \ i = 1,\ldots, k,
\end{equation}
where $c > 0$ and $\bar C>1$ can be taken independent of $i$ as well as $r$. By our hypothesis \EQ{noh4} and Lemma~\ref{mbounds_1} we deduce that $m_i(r) \leq C$ for all $r> 2r_0$ and all $i$. Thus \EQ{keyineq} implies that
\begin{equation*}
\inf\left\{ f_i(s) : m_{i}(r) \leq s \leq \bar C m_{i}(r) \right\} \leq C r^{-p} m_{i+1} (r)^{p-1} \rightarrow 0
\end{equation*}
as $r \to \infty$. Since $f_{i}$ is positive and continuous on $(0,\infty)$, and $m_{i}(r) \leq C$, we deduce that $m_{i}(r) \to 0$ for all $i$. Therefore \EQ{fsysloc} and \EQ{keyineq} imply that for all sufficiently large $r> 2r_0$ and each $i$,
\begin{equation*}
m_{i+1}(r) \geq c r^{\frac{p}{p-1}} m_{i}(r)^{\frac{\sigma_i}{p-1}}
\end{equation*}
By induction we have for sufficiently large $r>2r_0$,
\begin{align*}
m_1(r) & \geq c r^{\frac{p}{p-1}} m_k(r)^{\frac{\sigma_k}{p-1}} \\
& \geq c r^{\frac{p}{p-1}} \left( c r^{\frac{p}{p-1}}m_{k-1}(r)^{\frac{\sigma_{k-1}}{p-1}} \right)^{\frac{\sigma_k}{p-1}}  = c r^{\frac{p}{p-1}\left( 1+ \frac{\sigma_k}{p-1} \right)} m_{k-1}(r)^{\frac{\sigma_k\sigma_{k-1}}{(p-1)^2}} \\
& \geq c r^{\frac{p}{p-1}\left( 1+ \frac{\sigma_k}{p-1} \right)} \left( c r^{\frac{p}{p-1}} m_{k-2}(r)^{\frac{\sigma_{k-2}}{p-1}} \right)^{\frac{\sigma_k\sigma_{k-1}}{(p-1)^2}} \\
& \ldots \\
& \geq c r^A m_1(r)^B
\end{align*}
where we have written
\begin{equation*}
A : = \frac{p}{p-1}\left( 1+ \sum_{i=0}^{k-2}\frac{\sigma_k\ldots\sigma_{k-i}}{(p-1)^{i+1}} \right), \qquad
B: = \frac{\sigma_k\sigma_{k-1} \cdots \sigma_1}{(p-1)^k}.
\end{equation*}
Observing that $A/(B-1) = \beta_1$ and rearranging the inequality above, we get
\begin{equation*}
m_1(r) \leq C r^{-\beta_1} \quad \mbox{for sufficiently large} \ r>2r_0.
\end{equation*}
Since we have the lower bound $m_1(r)\geq c r^{-\alpha_1},$ we deduce an immediate contradiction in the case that $\beta_1 > \alpha_1$.  Thus we may assume that $\beta_1 = \alpha_1$. Hence $r^Am_1(r)^B \geq c m_1(r)$. Thus the string of inequalities above may be reversed, that is, we have
\begin{equation*}
m_{i+1}(r) \leq C r^{\frac{p}{p-1}} m_i(r)^{\frac{\sigma_i}{p-1}}
\end{equation*}
for sufficiently large $r$ and all $i$. Using this for $i=k$, we discover that on the set $A_r: = \left\{ x\in B_2\setminus B_1 : m_k(r) \leq u_k(rx) \leq \bar C m_k(r) \right\}$ we have
\begin{equation*}
-Q_1[u_{1,r}] \geq r^p f_k(u_{k,r}) \geq c r^p m_k(r)^{\sigma_k} \geq  cm_1(r)^{p-1} \geq c r^{-\alpha_1(p-1)}.
\end{equation*}
Defining $v_{1,r} : = r^{\alpha_1} u_{1,r}$, we obtain
\begin{equation} \label{vineq}
-Q_1 [ v_{1,r} ] \geq c \quad \mbox{in} \ A_r.
\end{equation}
We now proceed as in the proof of Theorem \ref{FNL} to deduce from the inequality \EQ{vineq} that $r^{\alpha_1} m_1(r) \to \infty$ as $r\to \infty$, a contradiction.

Finally, notice that if $D\leq 0$, then we may simply replace $\sigma_1$ by a larger number so that $D> 0$, but $D$ is small enough that $\beta_1 > \alpha_1$. The hypothesis \EQ{fsysloc} weakens as $\sigma_1$ increases. We have proved the following theorem:

\begin{thm} \label{SYS}
Suppose that for each $i=1,\ldots,k$ the elliptic operator $Q_i[\cdot]$ satisfies the hypotheses (H1)-(H5), with constants $p_i$, $\alpha^*_i$ and $\widetilde\alpha^*_i$, and  that \EQ{noh4} holds. Let $\sigma_1,\ldots,\sigma_k > 0$ and $f_i$ satisfy \EQ{fsysloc}, $D$ be given by \EQ{D}, and $\beta_i$ be given by \EQ{betak}. Suppose that either $D\leq 0$, or
$
D > 0 $ and $\min_{1\leq i \leq k} (\alpha_i - \beta_i) \leq 0$.
Then the system \re{LEsys}
has no positive solution  in any exterior domain of $\R^n$.
\end{thm}

\bibliographystyle{plain}
\bibliography{bigliouv1}

\end{document}